\documentclass[11pt]{article}

\usepackage[margin=1in]{geometry}
\usepackage[T1]{fontenc}
\usepackage[utf8]{inputenc}
\usepackage{microtype}

\usepackage{amsmath,amssymb,amsfonts,amsthm,mathtools}
\usepackage{extarrows} 

\usepackage{enumitem}
\usepackage[normalem]{ulem}
\usepackage{xcolor}
\usepackage{todonotes}
\usepackage[showdeletions]{color-edits}

\addauthor[Yixuan]{yz}{blue}
\addauthor[Qiaomin]{qx}{red}

\usepackage{natbib}
\setcitestyle{authoryear,round}

\usepackage[
    unicode=true,
    pdfusetitle,
    bookmarks=true,
    bookmarksnumbered=true,
    bookmarksopen=false,
    breaklinks=true,
    pdfborder={0 0 0},
    colorlinks=true,
    citecolor=blue,
    linkcolor=blue,
    urlcolor=blue
]{hyperref}

\newcommand{\R}{\mathbb{R}}
\newcommand{\E}{\mathbb{E}}
\newcommand{\Pp}{\mathbb{P}}

\newcommand{\law}{\mathcal{L}}
\newcommand{\W}{\mathcal{W}}

\newcommand{\TV}{\mathrm{TV}}
\newcommand{\HS}{\mathrm{HS}}
\newcommand{\one}{\mathbf{1}}

\newcommand{\norm}[1]{\left\lVert #1\right\rVert}
\newcommand{\abs}[1]{\left\lvert #1\right\rvert}
\newcommand{\ip}[2]{\left\langle #1,#2\right\rangle}

\newcommand{\ot}{\otimes}
\newcommand{\contr}{\mathbin{\lrcorner}}

\newtheorem{theorem}{Theorem}
\newtheorem{lemma}{Lemma}
\newtheorem{proposition}{Proposition}
\newtheorem{corollary}{Corollary}
\theoremstyle{definition}
\newtheorem{assumption}{Assumption}
\theoremstyle{remark}
\newtheorem{remark}{Remark}

\allowdisplaybreaks

\title{Gaussian Approximation in Wasserstein-$p$: Markovian Martingale Differences and Stochastic Approximation}

\author{
Yixuan Zhang and Qiaomin Xie
\\[0.8em]
Department of Industrial \& Systems Engineering, University of Wisconsin--Madison
\\
\texttt{\{yzhang2554,qiaomin.xie\}@wisc.edu}
}

\date{}

\begin{document}

\maketitle

\begin{abstract}
We establish nonasymptotic Gaussian approximation bounds in higher-order Wasserstein distance $\W_p$, $p\geq2$, for deterministic matrix-weighted sums of multivariate Poisson martingale differences generated by uniformly ergodic Markov chains. Under an $L^{(2+\eta)p}$-moment condition, $\eta>0$, we obtain the explicit bound
$$\mathcal O\left(p^3\|A\|_4^2+pd^{1/4}\|A\|_2^{1/2}\|A\|_4^2\right),$$
where $A\in\mathbb{R}^n$ collects the moment sizes of the weighted increments. When these sizes are of order $n^{-1/2}$, the bound achieves the first optimal $\mathcal O(n^{-1/2})$ rate for fixed $p$ and $d$. Consequently, we also obtain the first optimal
$\mathcal{O}(n^{-1/2})$ $\mathcal{W}_p$ Gaussian approximation rate for multivariate
additive functionals of uniformly ergodic Markov chains.

Our analysis develops two techniques for addressing the interplay between
higher-order Wasserstein distance and temporal dependence. First, building on
the Ornstein--Uhlenbeck relative-score approach of \citet{fang2023p}, we
formulate the bound in terms of antisymmetric Stein couplings while retaining
the conditional tensor structure. Second,
we develop a refresh-then-maximal coupling that combines an independent
first-step resampling, which preserves the desired Stein identity, with a
subsequent maximal coupling that provides effective control of the coupling
increment. 

We apply these bounds to nonlinear Markovian stochastic approximation by decomposing normalized errors into matrix-weighted Poisson martingale sums and remainders. This yields the first sharp $\W_p$ Gaussian approximation rates for both last iterates and Polyak--Ruppert averages. In particular, the optimized rates are $\mathcal O(n^{-1/3})$ for polynomial-stepsize last iterates and $\mathcal O(n^{-1/4})$ for Polyak--Ruppert averages, with an $\mathcal O(n^{-1/2})$ rate for harmonic-stepsize last iterates under an additional spectral condition.
\end{abstract}


\medskip

\section{Introduction}\label{sec:Intro}

Many models in operations research generate data sequentially
from stochastic systems with temporal dependence. Prominent
examples include steady-state simulation and Markov chain
Monte Carlo, stochastic approximation (SA) and optimization, and
policy evaluation in reinforcement learning. In these
settings, martingale decompositions provide a common way
to isolate the leading stochastic fluctuations of an
estimator. They also underlie the classical theory of
asymptotic normality and efficiency for Polyak--Ruppert
averaging in SA
\citep{polyak1992acceleration}.

The importance of these decompositions extends beyond
asymptotic theory. In finite-sample analyses of stochastic
algorithms, including stochastic-gradient and
temporal-difference methods, the estimation error is often
decomposed into a leading martingale component and remainder
terms. Combining quantitative Gaussian approximation for
the former with separate control of the latter yields
finite-sample guarantees for distributional approximation
and uncertainty quantification
\citep{anastasiou2019normal,srikant2025rates,
samsonov2024gaussian,wu2025uncertainty,kong2026finite}.
A central ingredient is therefore a sharp Gaussian
approximation bound for multivariate martingale sums.

To illustrate this approximation problem, consider an
additive functional of a Markov chain.
Let $(X_i)_{i\geq0}$ be a stationary uniformly ergodic
Markov chain on a state space $\mathsf X$, with transition
kernel $P$ and stationary distribution $\pi$.
Let $h:\mathsf X\to\mathbb R^d$ satisfy $\pi(h)=0$,
and let $g$ solve the Poisson equation $g-Pg=h$.
The Poisson equation gives the decomposition
\citep{glynn2024solution}
\begin{equation}
\label{eq:poisson-martingale-decomposition}
    \sum_{i=1}^n h(X_{i-1})
    =
    \sum_{i=1}^n
    \underbrace{
        \bigl[g(X_i)-Pg(X_{i-1})\bigr]
    }_{=:D_i}
    +g(X_0)-g(X_n).
\end{equation}
Here, $(D_i)_{i\geq1}$ is a martingale-difference sequence,
and $g(X_0)-g(X_n)$ is a telescoping boundary term.
Thus, Gaussian approximation for the normalized additive
functional can be obtained by approximating
$n^{-1/2}\sum_{i=1}^nD_i$ and controlling the boundary
term separately.

For illustration, suppose that
$\mathbb E[D_1D_1^\top]=I_d$.
By stationarity and martingale orthogonality, the normalized
martingale sum then has identity covariance.
We seek sharp bounds in higher-order Wasserstein distance,
$\mathcal W_p$ with $p\geq2$, on
\begin{equation}
\label{eq:main-gaussian-approximation-question}
    \mathcal W_p\!\left(
        \law\!\left(
            \frac1{\sqrt n}\sum_{i=1}^nD_i
        \right),
        \mathcal N(0,I_d)
    \right),
\end{equation}
where $\law(\cdot)$ denotes the law of a random vector.
For $p\geq1$, the Wasserstein distance of order $p$
between probability measures $\nu$ and $\mu$ on
$\mathbb R^d$ with finite $p$th moments is defined by
\citep{villani2008optimal}
\begin{equation}
\label{eq:wp}
    \mathcal W_p(\nu,\mu)
    :=
    \inf_{\gamma\in\Gamma(\nu,\mu)}
    \left(
        \mathbb E_{(X,Y)\sim\gamma}
        \bigl[\norm{X-Y}^p\bigr]
    \right)^{1/p},
\end{equation}
where $\Gamma(\nu,\mu)$ is the set of couplings of
$\nu$ and $\mu$.

Our focus on higher-order Wasserstein distance is motivated by several
applications. In the finite-sample analysis of stochastic algorithms,
\citet{kong2026finite} identify Gaussian approximation for martingale
differences in higher-order Wasserstein distance as a key missing ingredient
in the analysis of SA. Moreover, bounds in
\(\W_p\) with explicit dependence on the Wasserstein order \(p\) can serve as
useful inputs for studying Gaussian approximation in moderate- and
large-deviation regimes \citep{fang2023p,wang2026tail}.

Despite the relevance of this problem, sharp approximation guarantees
remain limited in the dependent settings arising in these applications.
Quantitative Gaussian approximation in higher-order Wasserstein distance
is well developed for independent sums, in both the scalar
\citep{bobkov2018berry} and multivariate settings
\citep{bonis2020stein,bonis2024improved}.
For martingale sums, earlier multivariate results primarily concern
smooth test-function metrics \citep{anastasiou2019normal}
or $\mathcal W_1$ \citep{srikant2025rates,wu2025uncertainty},
while \citet{guo2026rate} establishes higher-order Wasserstein bounds
for scalar martingales.
More recently, \citet{fang2026martingale} establish
$\mathcal W_p$ Gaussian approximation bounds for general
multivariate martingales.
Although these general bounds apply to the Poisson martingale sums
considered here, their direct specialization can yield suboptimal
rates; see the detailed comparison following
Theorem~\ref{thm:main}.
This gap motivates the development of sharp higher-order Wasserstein
bounds that exploit the underlying Markovian structure.

\subsection{Our Contributions}

\noindent\textbf{$\mathcal W_p$ Gaussian approximation for weighted
Poisson martingale sums.}
We study the more general matrix-weighted sums
\[
    S_n:=\sum_{i=1}^n M_{n,i}D_i,
    \qquad
    D_i:=g(X_i)-Pg(X_{i-1}),
\]
where the deterministic matrices $M_{n,i}\in\mathbb R^{d\times d}$
may depend on both the increment index $i$ and the terminal index $n$.
For a uniformly ergodic underlying Markov chain, we establish
a nonasymptotic Gaussian approximation bound in $\W_p$, $p\geq2$,
for these multivariate Poisson martingale sums.
Theorem~\ref{thm:main} expresses the bound in terms of the
$L^{(2+\eta)p}$-norms of the individual weighted increments
$M_{n,i}D_i$, for $\eta>0$, while explicitly tracking the dependence
on both the Wasserstein order $p$ and the dimension $d$.
In the balanced-increment regime, where these norms are of
comparable order $n^{-1/2}$, the bound yields the optimal
$O(n^{-1/2})$ rate for fixed $p$ and $d$.
To the best of our knowledge, this is the first bound to achieve
this optimal rate in higher-order Wasserstein distance for
this class of multivariate Markovian martingale sums.

\medskip
\noindent\textbf{Sharp Gaussian approximation for additive functionals.}
Taking $M_{n,i}=n^{-1/2}I_d$ recovers the normalized
unweighted sum in
\eqref{eq:main-gaussian-approximation-question}.
Combining our martingale bound with control of the Poisson
boundary term yields a Gaussian approximation bound for
multivariate additive functionals of Markov chains.
For fixed $p$ and $d$,
Corollary~\ref{cor:additive-functional} establishes the
optimal $O(n^{-1/2})$ rate in $\mathcal W_p$ under an
$L^{(2+\eta)p}$-moment condition.
This improves the $O(n^{-1/4})$ rate available from prior
higher-order Wasserstein results under weaker
geometric-ergodicity conditions
\citep{zhang2026wasserstein}.

\medskip
\noindent\textbf{Sharp Gaussian approximation for stochastic
approximation.}
Applying our bounds to matrix-weighted Poisson martingale
sums in SA error decompositions yields,  to the best of our knowledge, the first
sharp $\mathcal W_p$ Gaussian approximation bounds for both
last iterates and Polyak--Ruppert averages of Markovian SA.

For fixed $p\geq2$, $d$, and other problem parameters,
Theorem~\ref{thm:sa-application} gives an
$O(\sqrt{\gamma_n}+(n\gamma_n)^{-1})$ bound for normalized
last iterates with $\gamma_n=\gamma_1n^{-a}$, $0<a<1$.
For harmonic stepsizes, an additional spectral condition
yields the optimal $O(n^{-1/2})$ rate.
For PR averaging with $1/2<a<1$, the bound is
$O((n\gamma_n)^{-1}+\sqrt n\,\gamma_n)$, yielding
$O(n^{-1/4})$ at $a=3/4$.
These rates are sharp in sample size for the specified
iterates and centered limiting Gaussian targets.

\textit{The key is to retain the matrix weights}, first matching
the finite-sample Gaussian covariance and then comparing
it with its limit. This avoids losses from blockwise
approximation or reduction to unweighted sums.
Compared with \citet{kong2026finite}, even when their
analysis receives an optimal inverse-square-root martingale
approximation input, we improve the leading last-iterate
term from $\gamma_n^{1/6}$ to $\sqrt{\gamma_n}$ and the
optimized averaging rate from $O(n^{-1/6})$ to $O(n^{-1/4})$.

\medskip
\noindent\textbf{Technical contributions.}
\begin{itemize}[leftmargin=8pt]

\item \textbf{A Tensor-Level Relative-Score Bound via Antisymmetric Stein Couplings.}
Building on the Stein-method framework of \citet{bonis2020stein},
\citet{fang2023p} develop higher-order Wasserstein bounds for Gaussian
approximation using generalized exchangeable pairs. Closely related to this
formulation, we work with an antisymmetric Stein coupling
\((W,W',G)\) satisfying
\[
    (W,W',G)
    \stackrel{\mathrm d}{=}
    (W',W,-G).
\]
Our first technical ingredient is to formulate the relative-score argument
underlying the proof of \citet[Theorem~7.1]{fang2023p} directly in terms of
the conditional tensors arising from an antisymmetric Stein coupling; see
Proposition~\ref{prop:OU-Stein}. Rather than invoking the final
absolute-moment bound of \citet{fang2023p}, we retain the conditional tensor
structure so that it can be exploited to obtain sharp bounds
in our subsequent Markovian analysis. This formulation is essential for
deriving our results under moment assumptions only slightly stronger than
order \(2p\); in contrast, a direct application of
\citet[Theorem~7.1]{fang2023p}
would naturally require control of moments of order at least \(4p\).

\item \textbf{A Refresh-Then-Maximal Coupling.}
We construct an antisymmetric Stein coupling tailored to Markov chain induced
martingale sums. The construction combines an independent sampling of the
first transition with a symmetric blockwise maximal coupling of the
subsequent trajectories.  The refreshed first transition generates the
auxiliary variable \(G\), whereas the
subsequent maximal-coupling stage gives a geometrically decaying meeting time
and thereby  controll the coupling increment \(\Delta\). In addition, the construction preserves prefix measurability, which allows
us to exploit the mixing properties of the Markov chain to control the
quantities appearing in Proposition~\ref{prop:OU-Stein}.

\end{itemize}

\subsection{Organization and Notation}

The remainder of the paper is organized as follows.
Section~\ref{sec:problem-setup} introduces the Poisson equation
and the deterministic matrix-weighted martingale sums.
Section~\ref{sec:main-results} presents our main Gaussian
approximation result and its consequence for additive
functionals of Markov chains.
Section~\ref{sec:sa-application} applies the weighted Gaussian approximation
bound to last iterates and Polyak--Ruppert averages of stochastic approximation.
Section~\ref{sec:OU} develops a tensor-level refinement of the Ornstein--Uhlenbeck relative-score bound for antisymmetric Stein couplings.
Section~\ref{sec:coupling} introduces the refresh-then-maximal coupling for
Markov chains and verifies that it yields the antisymmetric Stein coupling
required by our framework.
Section~\ref{sec:completion} proves the main theorem.
The appendix contains the remaining proofs and the SA lower bounds. 

Throughout the paper, \(\norm{\cdot}\) and \(\langle\cdot,\cdot\rangle\) denote the
Euclidean norm and inner product, respectively. For matrices and tensors,
\(\norm{\cdot}_{\mathrm{op}}\) and \(\norm{\cdot}_{\HS}\) denote the operator
and Hilbert--Schmidt norms, with $\norm{x_1\ot\cdots\ot x_k}_{\HS}
    =
    \prod_{j=1}^k \norm{x_j}.$ We write \(x^{\ot k}\) for the \(k\)-fold tensor product of \(x\) with itself,
and \(A\contr B\) for contraction over matching tensor indices. We denote by $\operatorname{Sym}^k(\R^d)$ the space of symmetric \(k\)-tensors on \(\R^d\), that is, tensors
\(T\in(\R^d)^{\ot k}\) whose entries are invariant under permutations of their
indices. For a deterministic vector \(a=(a_1,\ldots,a_n)\in\R^n\), we write its $L^r$ norm as $\norm{a}_r
    :=
    \left(
        \sum_{i=1}^n |a_i|^r
    \right)^{1/r}$ for any $1\leq r<\infty$ and $\norm{a}_\infty:=\max_{1\leq i\leq n}|a_i|.$

For an \(\mathsf H\)-valued random variable \(X\), where \(\mathsf H\) is a
finite-dimensional Hilbert space, write $\norm{X}_{L^r(\mathsf H)}
    :=
    \bigl(\E[\norm{X}_{\mathsf H}^r]\bigr)^{1/r}.$ We omit \(\mathsf H\) when the ambient norm is clear; in particular,
\(\norm{X}_{L^r}:=(\E[\norm{X}^r])^{1/r}\) for \(X\in\R^d\), and
\(\norm{T}_{L^r(\HS)}:=(\E\norm{T}_{\HS}^r)^{1/r}\) for random tensors.
A subscript on \(L^r\), as in \(L_Z^r\), indicates that the norm is taken
only over the specified randomness $Z$. For measurable \(f:\mathsf X\to\mathsf H\), write $\norm{f}_{L^r(\pi;\mathsf H)}
    :=
    \left(\int_{\mathsf X}\norm{f(x)}_{\mathsf H}^r\,\pi(\mathrm dx)\right)^{1/r}$ and $
    \pi(f):=\int_{\mathsf X}f(x)\,\pi(\mathrm dx).$ If \(\norm{f}_{L^r(\pi;\mathsf H)}<\infty\), we write
\(f\in L^r(\pi;\mathsf H)\). When the ambient Hilbert space
\(\mathsf H\) is clear from context, we abbreviate
\(L^r(\pi;\mathsf H)\) as \(L^r(\pi)\).

For probability measures \(\mu\) and \(\nu\) on the same measurable space, $\norm{\mu-\nu}_{\TV}
    :=
    \sup_A |\mu(A)-\nu(A)|$ denotes their total variation distance. Finally, \(C(\eta,C_0,\rho_0)\) denotes a generic finite positive constant,
possibly varying from line to line, that depends only on
\((\eta,C_0,\rho_0)\).

\section{Problem Setup}
\label{sec:problem-setup}

Let \((\mathsf X,\mathcal B)\) be a standard Borel space, and let
\(\{X_k\}_{k\geq0}\) be a stationary Markov chain with state space \(\mathsf X\), transition kernel
\(P\), and stationary distribution \(\pi\).  Unless stated otherwise, all
expectations and covariances are taken under the stationary distribution \(\pi\). We assume
uniform ergodicity of the Markov chain \(\{X_k\}_{k\geq0}\)~\citep{meyn2012markov}.
\begin{assumption}\label{assumption:MC}
There exist \(C_0<\infty\) and
\(\rho_0\in(0,1)\) such that
\begin{equation}
\label{eq:UGE}
    \sup_{x\in\mathsf X}
    \norm{P^m(x,\cdot)-\pi}_{\TV}
    \leq C_0\rho_0^m,
    \qquad m\geq 0.
\end{equation}
\end{assumption}
Uniform ergodicity is widely used in operations research to analyze stochastic
systems driven by Markovian dynamics. Representative examples include inventory
control \citep{huh2009adaptive,tang2024online}, stochastic approximation and reinforcement
learning under Markovian data
\citep{chen2021lyapunov,zhang2024constant,li2024q}. In these settings, uniform
geometric forgetting of the initial state provides a  quantitative
stability condition for establishing finite-time performance guarantees.

The following lemma justifies the existence of a solution to the Poisson equation.

\begin{lemma}
\label{lem:Lr-Poisson}
Fix \(r\in(1,\infty)\). Under Assumption \ref{assumption:MC}, for every
\(m\geq0\) and every \(f\in L^r(\pi)\), 
\begin{equation}
\label{eq:Lr-uniform-decay}
    \norm{(P^m-\Pi)f}_{L^r(\pi)}
    \leq
    2\bigl(C_0\rho_0^m\bigr)^{1-1/r}
    \norm{f}_{L^r(\pi)},
\end{equation}
where \(\Pi f:=\pi(f)\).
Consequently, if \(f\in L^r(\pi)\) and \(\pi(f)=0\), then  $ \sum_{m=0}^N P^m f$ converge in \(L^r(\pi)\), as \(N\to\infty\), to a function
\(g_f\in L^r(\pi)\) satisfying $g_f-Pg_f=f$ and $\pi(g_f)=0.$
\end{lemma}

Fix $p\geq2$ and $\eta>0$, and set $q:=(2+\eta)p$.
Let $h\in L^q(\pi;\R^d)$ satisfy $\pi(h)=0$, and define
\[
    g:=\sum_{m=0}^{\infty}P^m h.
\]
By Lemma~\ref{lem:Lr-Poisson}, $g\in L^q(\pi)$ and  $g-Pg=h$. Let $\mathcal F_i:=\sigma(X_0,\ldots,X_i)$.
For deterministic matrices
$M_{n,1},\ldots,M_{n,n}\in\R^{d\times d}$, define
\begin{equation}
    D_i:=g(X_i)-Pg(X_{i-1}),
    \qquad
    Y_i:=M_{n,i}D_i,
    \qquad
    S_n:=\sum_{i=1}^nY_i.
\end{equation}
The Markov property gives
$\E[D_i\mid\mathcal F_{i-1}]=0$.
Thus, $\{Y_i\}_{i=1}^n$ is a martingale-difference sequence
with respect to the filtration $\{\mathcal F_i\}_{i=0}^n$.

\begin{remark}[Initialization]
\label{rem:initialization}
The quantitative Gaussian approximation with stationary initialization can be
extended to an arbitrary initial distribution satisfying suitable moment
conditions, by a standard maximal coupling argument; see, e.g., the discussion
in \cite[EC.5.2]{zhang2026wasserstein}. For the ease of exposition, we
work throughout with the stationary setting. 
\end{remark}

\begin{remark}[Matrix weights and applications]
\label{rem:matrix-weights}
The deterministic weights describe how each noise increment
propagates through an SA recursion. Retaining them permits
direct approximation of the complete weighted martingale sum
in the error decomposition. Section~\ref{sec:sa-application} uses this observation to
obtain sharp SA approximation rates.

Taking $M_{n,i}=n^{-1/2}I_d$ recovers the normalized
unweighted Poisson martingale sum. The Poisson decomposition
gives
\[
    \frac1{\sqrt n}\sum_{i=1}^n h(X_i)
    =
    \frac1{\sqrt n}\sum_{i=1}^nD_i
    +\frac{Pg(X_0)-Pg(X_n)}{\sqrt n}.
\]
By stationarity, the boundary term has $L^p$ norm at most
$2n^{-1/2}\norm{g}_{L^p(\pi)}$.
Combining this estimate with Theorem~\ref{thm:main} yields
the Gaussian approximation for Markov chain additive
functionals in Corollary~\ref{cor:additive-functional},
under the covariance normalization stated there.
\end{remark}

\section{Main Results}
\label{sec:main-results}

We now state our main result. For each
\(1\leq i\leq n\), define
\[
    a_i
    :=
    \norm{Y_i}_{L^q},
    \qquad
    A:=(a_1,\ldots,a_n)\in \R^n.
\]
Thus, $a_i=\norm{M_{n,i}D_i}_{L^q}$ records the size of
the weighted increment. The following theorem bounds the Gaussian
approximation error in terms of these increment sizes,
while retaining the dependence on $p$ and $d$.

\begin{theorem}[Gaussian Approximation for Weighted Poisson Martingale Sums]
\label{thm:main}
Under the setup of Section~\ref{sec:problem-setup}, Assumption~\ref{assumption:MC},
and the normalization \(\operatorname{Cov}(S_n)=I_d\),
\begin{align}
\label{eq:main-bound}
\W_p\!\left(
    \law(S_n),
    \mathcal N(0,I_d)
\right)
\leq
C(\eta,C_0,\rho_0)
\Big(
&
    p^3\norm{A}_4^2
    +
    p\,d^{1/4}
    \norm{A}_2^{1/2}
    \norm{A}_4^2
\Big).
\end{align}
\end{theorem} 
\noindent\textbf{Sample size dependence and optimality.}
Consider the balanced-increment regime
$\max_{i\leq n}a_i=O(n^{-1/2})$ for fixed $p$ and $d$.
Then
\[
    \norm{A}_2=O(1),
    \qquad
    \norm{A}_4^2=O(n^{-1/2}),
\]
and Theorem~\ref{thm:main} yields
\[
    \mathcal W_p\!\left(
        \law(S_n),\mathcal N(0,I_d)
    \right)
    =O(n^{-1/2}).
\]
This rate is optimal in general: normalized i.i.d.\ sums
are included by taking $P(x,\cdot)=\pi(\cdot)$,
$g=h$, and $M_{n,i}=n^{-1/2}I_d$.
The result also supplies the sharp higher-order Wasserstein
Gaussian approximation input for stochastic approximation
identified by \citet{kong2026finite}.

\medskip
\noindent\textbf{Comparison with existing work.}
Earlier multivariate results include smooth test-function bounds
\citep{anastasiou2019normal} and $\mathcal W_1$ bounds
\citep{srikant2025rates,wu2025uncertainty}.
For scalar martingales, \citet{guo2026rate} develops
higher-order Wasserstein bounds without imposing Markovian
structure. In the scalar balanced-increment regime of our
Markov setting, a direct specialization of Theorems~7 and~8
in that work yields $O(n^{-1/4})$ in $\mathcal W_2$ and
$O(n^{-1/6})$ in $\mathcal W_3$, even under the stronger
boundedness condition
$\max_{i\leq n}\norm{Y_i}_{L^\infty}=O(n^{-1/2})$.
Thus, exploiting Markovian structure improves the available
sample-size dependence even in one dimension.

More recently, \citet{fang2026martingale} establish
$\mathcal W_p$ bounds for general multivariate martingales,
with explicit terms accounting for fluctuations of the
\emph{predictable quadratic variation}
\[
    \mathcal V_n
    :=
    \sum_{i=1}^n
    \mathbb E[Y_iY_i^\top\mid\mathcal F_{i-1}].
\]
For fixed $p$ and $d$, under our covariance normalization
and  
$\max_{i\leq n}\norm{Y_i}_{L^{3p}}=O(n^{-1/2})$,
their Corollary~2.1 yields an $O(n^{-1/2}\log n)$ bound
provided that
\begin{equation}\label{eq:strong-condition}
    \mathcal V_n=I_d
    \quad\text{a.s.},
    \qquad
    \mathbb E[Y_iY_i^\top\mid\mathcal F_{i-1}]
    \succeq \frac{c}{n}I_d
    \quad\text{a.s. for every }i,
\end{equation}
where $c>0$ is independent of $n$.
Without these additional assumptions, a direct specialization
of their Theorem~2.3(ii) to our uniformly ergodic Markov
setting yields an $O(n^{-1/6})$ bound under the same
$L^{3p}$-moment condition.
In contrast, our theorem exploits the Markovian structure
and uniform ergodicity to obtain sharp bounds directly
in terms of marginal increment sizes, without requiring
deterministic predictable quadratic variation or conditional
ellipticity. It also permits the weaker
$L^{(2+\eta)p}$-moment condition for arbitrarily small
$\eta>0$.

\medskip
\noindent\textbf{Wasserstein order dependence.}
Theorem~\ref{thm:main} provides an explicit dependence on the Wasserstein order
\(p\), although we do not claim that the dependence on $p$ is optimal.
In the more restrictive balanced-increment regime $\sup_{p\geq 2}\max_{1\leq i\leq n}a_i\in O(n^{-1/2}),$ for fixed \(d\), our bound simplifies to
\[
    \W_p
    \lesssim
    \frac{p^3}{\sqrt n}.
\]
For comparison, for standardized sums of independent sub-exponential random
vectors, \citet{fang2023p} obtain
\[
    \W_p
    \lesssim
    \frac{p}{\sqrt n}
    +
    \frac{p^{5/2}}{n},
\]
which yields an essentially linear dependence \(O(p/\sqrt n)\) over the range
\(p\lesssim n^{1/3}\). \citet{fang2023p} further highlight that an explicit
polynomial control in \(p\) is a key ingredient in deriving sharp
Cram\'er-type moderate-deviation bounds. In comparison, the polynomial
dependence on \(p\) in Theorem~\ref{thm:main} is generally weaker.
Nevertheless, our analysis provides a starting point for developing
analogous results for Markov-chain-induced martingale differences. Sharpening
the  $p$ dependence, and determining whether the near-linear scaling attainable
in the independent setting can also be achieved under Markovian dependence,
remain interesting directions for future work.

\medskip
\noindent\textbf{Dimension dependence.}
Theorem~\ref{thm:main} makes part of the dimension dependence explicit
through the factor \(d^{1/4}\). We emphasize, however, that the overall
dependence on \(d\) is also implicit in the increment sizes
\(a_i=\norm{Y_i}_{L^q}\), and hence in \(A=(a_1,\ldots,a_n)\), as is also
the case in related bounds for independent sums \citep{bonis2020stein}.
In fact, the optimal dimension dependence remains unclear even for sums of
independent random vectors; we refer to the discussion following Theorem~1 of
\citet{zhang2026wasserstein} for further details.

\medskip
\noindent\textbf{\(L^{(2+\eta)p}\)-moment condition.}
The assumption $h\in L^q(\pi;\R^d)$, with
$q=(2+\eta)p$, ensures $g,D_i,Y_i\in L^q$ in their
respective spaces. Since $q>2p$, we can choose an intermediate
exponent $s\in(p,q/2)$. The gap \(s>p\) allows H\"older's inequality to convert mixing into \(L^p\)-decay; see \eqref{eq:beta-projection}. Meanwhile, the condition \(2s<q\) provides the additional integrability needed to control products of coupled increments in \(L^s\) while preserving the decay from the mixing; see \eqref{eq:suffix-basic-moments}.  Thus, the current argument requires \(q>2p\). We adopt the multiplicative slack
\(q=(2+\eta)p\) so that the exponent differences
\(1/p-1/s\) and \(1/(2s)-1/q\) are both of order \(1/p\), which is important
for tracking the dependence on \(p\). In contrast, an additive
slack of the form \(q=2p+\eta\) would make these exponent differences of
order \(1/p^2\), leading to a worse dependence on \(p\) in our current
analysis. If one does not seek explicit control of the Wasserstein-order
dependence, however, the same argument continues to hold under the weaker
\(L^{2p+\eta}\)-moment condition. This modification affects only the
dependence on \(p\) and does not change the resulting dependence on  \(n\) or  \(d\). Further details
are given in Section~\ref{sec:completion}. It remains open whether comparable bounds hold
under weaker moment conditions.

\medskip
\noindent\textbf{Proof roadmap.}
Section~\ref{sec:antisymmetric} formulates the relative-score bound used for
\(\W_p\) Gaussian approximation in terms of antisymmetric Stein couplings
and conditional tensors. Section~\ref{sec:coupling} then introduces the refresh-then-maximal coupling
tailored to Markov-chain-induced martingale sums.
Finally, Section~\ref{sec:completion} develops the required conditional-tensor
estimates and combines these ingredients to prove
Theorem~\ref{thm:main}.

\subsection{Implication for Additive Functionals of Markov Chains}
\label{sec:additive-functional-consequence}

We now take equal weights $M_{n,i}=n^{-1/2}I_d$ and
combine the resulting martingale bound with the Poisson
boundary term. Let $h$ and $g$ be as in
Section~\ref{sec:problem-setup}, and recall
\begin{equation}
\label{eq:poisson-differences-special-case}
    D_i=g(X_i)-Pg(X_{i-1}),
    \qquad
    \mathsf S_n:=\sum_{i=1}^n h(X_i).
\end{equation}
The existence of the asymptotic covariance of \(\mathsf S_n\) is established
in the literature; see, e.g., \citet[Lemma EC.2]{zhang2026wasserstein}. For simplicity,
we work under the normalization 
\begin{equation}
\label{eq:additive-asymptotic-covariance}
    \lim_{n\to\infty}
    \frac{1}{n}
    \operatorname{Cov}(\mathsf S_n)
    =
    I_d.
\end{equation}
The Poisson decomposition expresses \(\mathsf S_n\) as the martingale sum
covered by Theorem~\ref{thm:main}, together with a telescoping boundary term
that can be controlled separately. This yields the following corollary, whose
proof is deferred to Section~\ref{sec:proof-additive-functional}.
\begin{corollary}[Gaussian Approximation for Additive Functionals of Markov Chains]
\label{cor:additive-functional}
Under Assumption~\ref{assumption:MC},
let $h\in L^q(\pi;\mathbb R^d)$ satisfy $\pi(h)=0$. Under
the normalization \eqref{eq:additive-asymptotic-covariance}, for every $n\geq1$,
\begin{equation*}
\W_p\!\left(
    \law\!\left(
        \mathsf S_n/\sqrt n
    \right),
    \mathcal N(0,I_d)
\right)
\leq
\frac{
    C(\eta,C_0,\rho_0)
}{\sqrt n}
\left[
    \norm{h}_{L^q(\pi)}
    +
    p^3\norm{h}_{L^q(\pi)}^2
    +
    p\,d^{1/4}
    \norm{h}_{L^q(\pi)}^{5/2}
\right].
\end{equation*}
\end{corollary}

Prior work on \(\W_p\) (\(p\geq2\)) Gaussian approximation for sums of multivariate
additive functionals of Markov chains achieves only an
\(\mathcal{O}(n^{-1/4})\) rate, under weaker geometric-ergodicity conditions
\citep{zhang2026wasserstein}. In comparison, for fixed \(p\) and \(d\), our
result establishes the optimal \({O}(n^{-1/2})\) rate under an
\(L^{(2+\eta)p}\)-moment condition.

\section{Application to Markovian Stochastic Approximation}
\label{sec:sa-application}
The normalized errors of Markovian SA decompose into deterministic matrix-weighted sums of
Poisson martingale differences and remainder terms.
Theorem~\ref{thm:main} applies directly to these weighted
sums, without approximating the driving noise block by block
or replacing the weights by their limiting values.
We use this approach to obtain sharp Gaussian approximation
rates in $\W_p$, $p\geq2$, for both last iterates and
Polyak--Ruppert averages under the assumptions below.

\subsection{Setup for Markovian SA}

Let $(X_k)_{k\geq0}$ be the stationary Markov chain introduced
in Section~\ref{sec:problem-setup}, satisfying
Assumption~\ref{assumption:MC}. As discussed in
Remark~\ref{rem:initialization}, the stationarity assumption
can be relaxed under suitable moment conditions on the
initial distribution. 
Consider the Markovian SA recursion
\begin{equation*}
    \theta_{k+1}
    =\theta_k-\gamma_k F(\theta_k,X_k),
    \qquad
    \gamma_k=\gamma_1 k^{-a},
    \qquad k\geq1,
\end{equation*}
where $\gamma_1>0$, $a\in(0,1]$, and the initial iterate
$\theta_1$ is deterministic. Define the mean field
\[
    \overline F(\theta)
    :=\int_{\mathsf X}F(\theta,x)\,\pi(\mathrm dx).
\]
The recursion aims to approximate the unique solution
$\theta^*\in\R^d$ to $\overline F(\theta^*)=0$.
Fix $p\geq2$ and $\eta>0$, and set $q:=(2+\eta)p$.
We impose the following regularity and moment-stability
conditions.

\begin{assumption}
\label{ass:sa-application}
The following conditions hold.
\begin{enumerate}[label=\textup{(\roman*)}]
\item
The function $h:=F(\theta^*,\cdot)$ belongs to
$L^q(\pi;\R^d)$.

\item
There exists $L_0<\infty$ such that
\[
    \norm{F(\theta,x)-F(\theta',x)}
    \leq L_0\norm{\theta-\theta'},\qquad \forall \theta,\theta'\in\R^d,x\in\mathsf X.
\]

\item
There exist $H\in\R^{d\times d}$, $\delta>0$, and
$L_1<\infty$ such that
\[
    \norm{\overline F(\theta)-H(\theta-\theta^*)}
    \leq L_1\norm{\theta-\theta^*}^2,
    \qquad
    \norm{\theta-\theta^*}\leq\delta.
\]
In particular, $H=\nabla\overline F(\theta^*)$.
All eigenvalues of $H$ have strictly positive real parts.

\item
Let $g:=\sum_{m=0}^{\infty}P^m h$ be the solution to the
Poisson equation $g-Pg=h$, and define
\[
    D_k:=g(X_k)-Pg(X_{k-1}),
    \qquad
    \Gamma:=\E[D_1D_1^\top].
\]
The covariance matrix $\Gamma$ is positive definite.

\item
The iterates satisfy the moment-stability condition
\begin{equation}
\label{eq:sa-moment-stability}
    K_{2p}
    :=\sup_{k\geq1}
    \gamma_k^{-1/2}
    \norm{\theta_k-\theta^*}_{L^{2p}}
    <\infty.
\end{equation}
\end{enumerate}
\end{assumption}
Assumption~\ref{ass:sa-application} collects regularity,
nondegeneracy, and moment-stability conditions for Gaussian
approximation of Markovian SA.
The local expansion in condition~\textup{(iii)} is exact
for linear SA, including linear TD learning.
For nonlinear SA, we require only a local linearization
of the mean field $\overline F$ with a quadratic remainder.
In contrast, \citet{kong2026finite} require
$\nabla_\theta F(\theta,x)$ to be globally Lipschitz in
$\theta$, uniformly over $x$.
Our formulation therefore accommodates globally nonsmooth
updates, including asynchronous $Q$-learning under suitable
local regularity conditions.
In particular, conditions~\textup{(i)--(iii)} hold for the
finite-state asynchronous $Q$-learning  model studied by \citet{liu2025central}.

Condition~\textup{(iv)} ensures nondegenerate Gaussian
targets.
Analogous nondegeneracy conditions appear in
\citet[Theorem~4.3]{wu2025uncertainty} for Markovian TD
learning and \citet[Assumption~A2]{samsonov2025statistical}
for Markovian linear SA.
Although \citet{liu2025central} do not explicitly impose
this condition in their main theorem, the martingale CLT
invoked in Appendix~A.3 requires a positive definite
limiting covariance.

Finally, condition~\textup{(v)} requires the quantitative
moment bound
\eqref{eq:sa-moment-stability},
rather than merely finiteness of the moments.
Finite-time moment bounds for Markovian linear SA and
TD learning have been established by
\citet{durmus2021stability}; see also
\citet[Section~2.3]{kong2026finite} and the references therein.
 We impose it directly to separate
Gaussian approximation from application-specific stability
and moment analysis.

\subsection{Last-Iterate and Averaged-Iterate Bounds}

We now present Gaussian approximation bounds in $\W_p$,
$p\geq2$, for both last iterates and Polyak--Ruppert (PR)
averages of Markovian SA under different stepsize regimes.
Define the PR average \citep{polyak1992acceleration} by
\[
    \overline\theta_n:=\frac1n\sum_{j=1}^n\theta_{j+1},
\]
and let the covariance matrices $\Sigma$ and $\overline\Sigma$
be defined by
\begin{equation}
\label{eq:sa-limit-covariances}
    H\Sigma+\Sigma H^\top=\Gamma,
    \qquad
    \overline\Sigma=H^{-1}\Gamma H^{-\top}.
\end{equation}
Under Assumption~\ref{ass:sa-application}, both matrices
are well defined and positive definite.
The following theorem gives the approximation bounds;
its proof is provided in
Section~\ref{sec:proof-sa-application}.
\begin{theorem}[Gaussian approximation for stochastic approximation]
\label{thm:sa-application}
Under Assumptions~\ref{assumption:MC} and
\ref{ass:sa-application}, the following bounds hold.

\smallskip
\noindent\textup{(a) Polynomial stepsizes.}
If $0<a<1$, then
\begin{equation}
\label{eq:sa-last-rate}
    \W_p\!\left(
        \law\!\left(
            \frac{\theta_{n+1}-\theta^*}{\sqrt{\gamma_n}}
        \right),
        \mathcal N(0,\Sigma)
    \right)
    \lesssim
    \sqrt{\gamma_n}+\frac{1}{n\gamma_n}.
\end{equation}
In particular, $a=2/3$ gives an $O(n^{-1/3})$ bound.

\smallskip
\noindent\textup{(b) Harmonic stepsizes.}
If $a=1$ and
\begin{equation}
\label{eq:sa-harmonic-gap}
    \gamma_1\min_{\lambda\in\operatorname{spec}(H)}
        \operatorname{Re}\lambda>1,
\end{equation}
let $H_1:=H-(2\gamma_1)^{-1}I_d$ and let
$\Sigma_1\succ0$ solve
$H_1\Sigma_1+\Sigma_1H_1^\top=\Gamma$. Then
\begin{equation}
\label{eq:sa-harmonic-rate}
    \W_p\!\left(
        \law\!\left(
            \frac{\theta_{n+1}-\theta^*}{\sqrt{\gamma_n}}
        \right),
        \mathcal N(0,\Sigma_1)
    \right)
    \lesssim n^{-1/2}.
\end{equation}

\smallskip
\noindent\textup{(c) Polyak--Ruppert averaging.}
If $1/2<a<1$, then
\begin{equation}
\label{eq:sa-pr-rate}
    \W_p\!\left(
        \law\!\left(
            \sqrt n(\overline\theta_n-\theta^*)
        \right),
        \mathcal N(0,\overline\Sigma)
    \right)
    \lesssim
    \frac1{n\gamma_n}
    +\sqrt n\,\gamma_n.
\end{equation}
In particular, $a=3/4$ gives an $O(n^{-1/4})$ bound.
\end{theorem}

\medskip
\noindent\textbf{Scope of the rates.} The bounds in Theorem~\ref{thm:sa-application} focus on the
dependence on the sample size $n$, treating the Wasserstein
order $p$, the dimension $d$, and all other problem parameters
as fixed. Throughout this section, $\lesssim$ hides constants
that may depend on these parameters but not on $n$.
Explicit dependence on $p$ and $d$ could be obtained by
combining Theorem~\ref{thm:main} with a more refined analysis
of the stochastic approximation remainder and covariance terms.
We leave this refinement for future work.

\medskip
\noindent\textbf{Sharpness in sample size.}
The rates in Theorem~\ref{thm:sa-application} are
\emph{sharp} in their dependence on $n$, for the specified
iterates and centered limiting Gaussian targets.
For $0<a<1$, the last-iterate bound has order
\[
    n^{-a/2}+n^{-(1-a)}.
\]
A matching $\mathcal W_1$ lower bound for linear SA with
additive i.i.d.\ noise is established by
\citet[Proposition~4.5]{haque2026accurately}.
Since $\mathcal W_p\geq\mathcal W_1$, this lower bound
also applies to $\mathcal W_p$.
Balancing the two terms gives $a=2/3$ and the optimal
$n^{-1/3}$ rate among polynomial stepsizes with $0<a<1$.
For harmonic stepsizes, the $n^{-1/2}$ rate is also sharp,
as shown by Proposition~\ref{prop:sa-harmonic-lower-bound}.

For PR averaging, we consider the classical
stepsize regime $1/2<a<1$ \cite{polyak1992acceleration}, in which the bound has order
\[
    n^{-(a-1/2)}+n^{-(1-a)}.
\]
Proposition~\ref{prop:sa-pr-lower-bound} establishes
matching lower bounds in this setting.
Balancing the two terms gives $a=3/4$ and the optimal
$n^{-1/4}$ rate within this family of
Polyak--Ruppert averages.

\medskip
\noindent\textbf{Comparison with previous Markovian SA bounds.}
Earlier Wasserstein Gaussian approximation results for
Markovian SA primarily concern $\W_1$.
\citet{srikant2025rates} develops quantitative CLTs for
Markov chains and applies them to averaged TD learning.
For asynchronous $Q$-learning,
\citet[Theorem~4 and Corollary~5]{liu2025central} obtains
a $\widetilde O(n^{-1/6})$ Gaussian approximation bound
in $\W_1$ for the averaged iterates.

Complementary results for linear SA use other probability
metrics. \citet[Theorem~4.3]{wu2025uncertainty} establishes
an $O(n^{-1/4}\log n)$ bound in convex distance for
averaged TD learning with Markovian observations.
For Markovian linear SA,
\citet[Corollary~1 and Remark~1]{samsonov2025statistical}
obtains a $\widetilde O(n^{-1/4})$ bound in Kolmogorov
distance for one-dimensional projections of the averaged
iterates. These results address linear recursions and
do not directly provide higher-order Wasserstein bounds.

Turning to higher-order Wasserstein distance,
\citet{kong2026finite} develops a framework for nonlinear Markovian
SA that takes quantitative Gaussian approximation of the
driving martingale noise as an input.
Even when the normalized noise sums admit an
$O(m^{-1/2})$ approximation rate over $m$ observations,
their analysis yields a leading last-iterate term of
$\gamma_n^{1/6}$ and an optimized $O(n^{-1/6})$ rate
for PR averaging.
Under the uniform ergodicity, regularity, and
moment-stability assumptions imposed here,
Theorem~\ref{thm:sa-application} improves the
last-iterate bound to
$O(\sqrt{\gamma_n}+(n\gamma_n)^{-1})$ for $0<a<1$,
and the PR bound to $O(n^{-1/4})$ at $a=3/4$.
Both bounds hold in $\W_p$ for every fixed $p\geq2$.
In particular, for asynchronous $Q$-learning models
satisfying our assumptions, the averaged-iterate result improves the sample-size dependence of
\citet{liu2025central} while controlling the stronger
$\W_p$ metric.

The improvement over \citet{kong2026finite} and
\citet{liu2025central} comes from avoiding intermediate
approximations of the leading weighted martingale sum.
For last iterates, \citet{kong2026finite} approximates
the driving noise block by block; for averaging, both
analyses reduce the leading term to an unweighted
martingale sum.
Instead, we retain the complete matrix-weighted sum
$\sum_{k=1}^n M_{n,k}D_k$ obtained by linearizing the
mean field and unrolling the recursion.
Theorem~\ref{thm:main} approximates this sum directly
by a Gaussian with matching finite-sample covariance,
which is then compared with the limiting covariance.
For averaging, this reduces the weight-approximation
cost from $(n\gamma_n)^{-1/2}$ to $(n\gamma_n)^{-1}$.
Finite-sample response weights and Gaussian covariance
comparison also appear in
\citet{wu2025uncertainty,samsonov2025statistical};
our contribution is sharp multivariate $\W_p$
approximation for the resulting matrix-weighted
Markovian martingale sums and its application to
nonlinear SA.

\medskip
\noindent\textbf{Role of the harmonic spectral gap.}
Put $ \beta:=\gamma_1
    \min_{\lambda\in\operatorname{spec}(H)}
        \operatorname{Re}\lambda.$ The condition $\beta>1/2$ ensures that
$H_1=H-(2\gamma_1)^{-1}I_d$ is positive stable,
but does not guarantee the $O(n^{-1/2})$ Gaussian
approximation rate. The stronger condition $\beta>1$
is imposed to achieve this rate, rather than merely
to ensure asymptotic normality.
Proposition~\ref{prop:sa-harmonic-gap-obstructions}
demonstrates why this distinction matters.
When $1/2<\beta<1$, a linear example exhibits a
normalized initialization bias of order
$n^{-(\beta-1/2)}$; at the critical value $\beta=1$,
a nonlinear example exhibits a bias of order
$n^{-1/2}\log n$.
These examples show that the condition $\beta>1$
cannot, in general, be relaxed to $\beta>1/2$
while retaining the $O(n^{-1/2})$ rate.

\section{An Ornstein--Uhlenbeck Bound via Antisymmetric
Stein Couplings}
\label{sec:OU}
In this section, we build on the Ornstein--Uhlenbeck relative-score approach
of \citet{fang2023p} for Gaussian approximation in \(\W_p\), \(p\geq2\),
and formulate it in terms of antisymmetric Stein couplings and conditional
tensors.  Let
\(W\in L^p(\R^d)\) satisfy $\E[W]=0$ and $
    \E[WW^\top]=I_d.$ Let \(Z\sim\gamma_d:=\mathcal N(0,I_d)\) be $d$-dimensional standard Gaussian and independent of all other random
variables under consideration. For each $t>0,$ define $\alpha_t:=e^{-t},$ $\beta_t:=\sqrt{1-e^{-2t}},$ $\delta_t:=\frac{\beta_t}{\alpha_t}
    =\sqrt{e^{2t}-1},$ and 
\begin{equation}
\label{eq:OU}
W_t:=\alpha_tW+\beta_tZ.
\end{equation}
As shown in \cite[Section~6.1]{fang2023p}, \(\law(W_t)\) admits a smooth,
strictly positive density \(f_t\) with respect to \(\mathcal N(0,I_d)\).
Accordingly, the relative score of \(\law(W_t)\) with respect to
\(\mathcal N(0,I_d)\) is defined by
\[
    \mathfrak{s}_t(x)
    :=
    \nabla \log f_t(x).
\]
The following standard transport estimate converts control of the relative
score along the Ornstein--Uhlenbeck (OU) flow \eqref{eq:OU} into a bound on the Gaussian
approximation error in \(\W_p\). The
result is established in
\citet[Eq.~(6.13) and Section~8.2]{fang2023p}.
\begin{lemma}[OU Transport Bound]
\label{lem:OU-transport}
Let \(p\geq2\), and let \(W\) be an \(\R^d\)-valued random vector satisfying $\E\!\left[\norm{W}^p\right]<\infty.$ Then
\begin{equation}
\label{eq:OU-transport}
    \W_p\!\left(
        \law(W),
        \mathcal N(0,I_d)
    \right)
    \leq
    \int_0^\infty
        \norm{\mathfrak{s}_t(W_t)}_{L^p}
    \mathrm dt,
\end{equation}
where the right-hand side is understood as an improper integral with values
in \([0,\infty]\).
\end{lemma}
It therefore suffices to control \(\norm{\mathfrak{s}_t(W_t)}_{L^p}\). Section~\ref{sec:antisymmetric} establishes the relative-score bound via
antisymmetric Stein couplings, and Section~\ref{sec:proof-OU-Stein}
provides a self-contained proof.

\subsection{A Score Bound via Antisymmetric Stein Couplings} \label{sec:antisymmetric}

Stein couplings introduced by \citet{chen2010stein} provide a unifying
formulation for coupling-based Gaussian approximation. Their multivariate
extension was developed by \citet{fang2015rates}. Specifically, a triple of \(d\)-dimensional random vectors \((W,W',G)\) is a
multivariate Stein coupling if it satisfies
\begin{equation}
\label{eq:Stein-identity}
\E\!\left[
    G\cdot
    \bigl\{
        \varphi(W')-\varphi(W)
    \bigr\}
\right]
=
\E\!\left[
    W\cdot\varphi(W)
\right]
\end{equation}
for every bounded measurable vector field
 \(\varphi\). A key feature of this formulation is that the auxiliary variable \(G\)
is separated from the coupling increment
\(\Delta:=W'-W\) and, in particular, need not be proportional to it
\citep{chen2010stein}. 

We use the term \emph{antisymmetric Stein coupling} for a
multivariate Stein coupling that additionally satisfies
\begin{equation}
\label{eq:abstract-Stein-assumptions}
    (W,W',G)
    \stackrel{\mathrm d}{=}
    (W',W,-G).
\end{equation}
We will employ the above antisymetric Stein coupling to bound the OU relative score, as 
stated in the following Proposition \ref{prop:OU-Stein}.

\begin{proposition}[Truncated Antisymmetric Score Bound]
\label{prop:OU-Stein}
Let \(p\geq2\), and let \((W,W',G)\) be an antisymmetric Stein coupling such that
\[
    W,W',G\in L^p(\R^d),
    \qquad
    \E[W]=0,
    \qquad
    \E[WW^\top]=I_d.
\]
Fix \(t>0\). Let \(\mathsf A_t\subseteq(\mathbb R^d)^3\) be a measurable set
satisfying
\begin{equation}\label{eq:At_invariance}
    (w,w',g)\in\mathsf A_t
    \quad\Longleftrightarrow\quad
    (w',w,-g)\in\mathsf A_t,\qquad \forall w,w',g\in\mathbb R^d 
\end{equation}
and define the event $\mathcal A_t
    :=
    \{(W,W',G)\in\mathsf A_t\} $. Assume further that, on the event \(\mathcal A_t\),
\begin{equation}
\label{eq:small-increment}
\frac{\sqrt{p-1}}{\delta_t}\norm{\Delta}
\leq
\frac14.
\end{equation}
Define the \(W\)-measurable random tensors
\begin{align*}
R_t
&:=
2\E\!\left[
    G\one_{\mathcal A_t^c}
    \,\middle|\,
    W
\right],
\\
T_{k,t}
&:=
\E\!\left[
    G\ot\Delta^{\ot k}\one_{\mathcal A_t}
    \,\middle|\,
    W
\right],
\qquad k\geq1.
\end{align*}
Then
\begin{align}
\label{eq:score-bound-general}
\norm{\mathfrak{s}_t(W_t)}_{L^p}
\leq
e^{-t}\Bigg[\norm{R_t}_{L^p}
+
\frac{\sqrt{p-1}}{\delta_t}
\norm{T_{1,t}-I_d}_{L^p(\HS)}+
\left\{
\sum_{k=2}^{\infty}
\frac{(p-1)^k}{k!\,\delta_t^{2k}}
\norm{T_{k,t}}_{L^p(\HS)}^2
\right\}^{1/2}
\Bigg].
\end{align}
\end{proposition}

\noindent\textbf{Comparison with prior analysis.}
Our analysis builds on the Ornstein-Uhlenbeck relative-score approach of \citet{fang2023p}, leveraging antisymmetric Stein coupling.
In contrast, \citet[Theorem~7.1]{fang2023p} derive a related relative-score bound using
generalized exchangeable pairs, where the coupling is specified through the
conditional regression condition in \citet[Eq.~(7.1)]{fang2023p} rather than
through the Stein identity \eqref{eq:Stein-identity}. Although the two
formulations are closely related, the final bound in
\citet[Theorem~7.1]{fang2023p} reduces the higher-order conditional tensors
to conditional absolute moments. In particular, it involves a conditional
cubic quantity of the form
\[
    \norm{
        \E\!\left[
            \norm{G}\norm{\Delta}^3
            \,\middle|\,
            X
        \right]
    }_{L^p},
\]
where \(X\) denotes the underlying randomness with respect to which
\(W\) is measurable. If \(G\) and \(\Delta\) are controlled only through a common marginal
\(L^q\)-moment bound, then conditional Jensen's inequality and H\"older's
inequality give
\[
    \norm{
        \E\!\left[
            \norm{G}\norm{\Delta}^3
            \,\middle|\,
            X
        \right]
    }_{L^p}
    \leq
    \norm{\norm{G}\norm{\Delta}^3}_{L^p}
    \leq
    \norm{G}_{L^{4p}}
    \norm{\Delta}_{L^{4p}}^3.
\]
Thus, a direct application of the final bound naturally requires control of
moments of order at least \(4p\). In \citet{fang2023p}, the relative-score
bound is applied to obtain Gaussian approximation results under
sub-exponential tail assumptions, which in particular guarantee the
existence of moments of all orders.

For our setting, retaining the conditional tensor structure allows positive
and negative tensor contributions from different perturbations to cancel
before norms are taken, rather than bounding each contribution separately by
its magnitude. This additional cancellation is crucial for obtaining sharp
bounds under Markovian dependence and leads to moment requirements only
slightly stronger than order \(2p\).

\noindent\textbf{A moment identity implied by the Stein coupling.}
 Since \(p\geq2\), we have \(W,G,\Delta\in L^2\).
For every \(B\in\R^{d\times d}\), apply
\eqref{eq:Stein-identity} to $\varphi_R(w)
:=
B\operatorname{proj}_R(w),$ where \(\operatorname{proj}_R\) denotes Euclidean projection onto the closed
ball of radius \(R\). By nonexpansiveness of the projection,
\[
\abs{W\cdot\varphi_R(W)}
\leq
\norm{B}_{\mathrm{op}}\norm{W}^2,
\qquad
\abs{
G\cdot
\bigl\{
    \varphi_R(W')-\varphi_R(W)
\bigr\}
}
\leq
\norm{B}_{\mathrm{op}}\norm{G}\norm{\Delta}.
\]
Dominated convergence therefore yields
\[
\E[W^\top BW]
=
\E[G^\top B\Delta].
\]
Since this identity holds for every \(B\in\R^{d\times d}\), we have $\E[\Delta G^\top]
=
\E[WW^\top]
=
I_d.$ Taking transposes gives
\begin{equation}
\label{eq:Stein-covariance-calibration}
\E[G\ot\Delta]
=
\E[G\Delta^\top]
=
I_d.
\end{equation}
Thus, the Stein identity forces the mean of the tensor
\(G\ot\Delta\) to match the covariance matrix of \(W\), which is \(I_d\)
under our normalization. This identity will be used below to control
\(T_{1,t}-I_d\).

\subsection{Proof of Proposition~\ref{prop:OU-Stein}}
\label{sec:proof-OU-Stein}

The proof consists of three steps. First, we derive a truncated version of the Stein identity and combine it with the Gaussian shift formula to obtain a representation of the relative score \(\mathfrak{s}_t\); see \eqref{eq:truncated-Stein} and \eqref{eq:score-representation}. Second, we use the small-increment condition on
$\mathcal A_t$ to justify a  Hermite expansion  in $L^p$ for the key term in \(\mathfrak{s}_t\); see \eqref{eq:conditional-Hermite}. 
Finally,
we convert this expansion into the desired bound on the score.

\medskip
\noindent\textbf{Score
representation.}
By the invariance of \(\mathcal A_t\) and the antisymmetry \eqref{eq:abstract-Stein-assumptions}, \begin{align*} \E\!\left[ G\one_{\mathcal A_t^c}\cdot\varphi(W') \right] &= -\E\!\left[ G\one_{\mathcal A_t^c}\cdot\varphi(W) \right] \end{align*} for every bounded measurable vector field \(\varphi:\R^d\to\R^d\). Hence \begin{align*} \E\!\left[ G\one_{\mathcal A_t^c} \cdot \bigl\{ \varphi(W')-\varphi(W) \bigr\} \right] = -2\E\!\left[ G\one_{\mathcal A_t^c}\cdot\varphi(W) \right] = -\E\!\left[ R_t\cdot\varphi(W) \right], \end{align*} where the last equality follows from the definition $ R_t = 2\E\!\left[ G\one_{\mathcal A_t^c} \,\middle|\, W \right].$ Decomposing the Stein identity \eqref{eq:Stein-identity} over \(\mathcal A_t\) and \(\mathcal A_t^c\) therefore yields \begin{equation} \label{eq:truncated-Stein} \E\!\left[ (W+R_t)\cdot\varphi(W) \right] = \E\!\left[ G\one_{\mathcal A_t} \cdot \bigl\{ \varphi(W')-\varphi(W) \bigr\} \right]. \end{equation}
Set $\lambda_t:=\delta_t^{-1}.$ The standard relative-score representation for the OU flow
\citep[Lemma~5]{bonis2020stein} gives
\begin{equation}
\label{eq:relative-score}
    \mathfrak{s}_t(W_t)
    =
    \alpha_t
    \E\!\left[
        W-\lambda_t Z
        \,\middle|\,
        W_t
    \right].
\end{equation}
Let $\varphi:\R^d\to\R^d$ be bounded and measurable, and define $\psi(w)
    :=
    \E_Z\!\left[
        \varphi(\alpha_t w+\beta_t Z)
    \right].$ Since $W'=W+\Delta$ and $\lambda_t=\alpha_t/\beta_t$, the Gaussian shift
formula yields
\[
    \psi(W')
    =
    \E_Z\!\left[
        \varphi(W_t)L_t
    \right],
    \qquad
    L_t
    :=
    \exp\left\{
        \lambda_t\ip{\Delta}{Z}
        -\frac{\lambda_t^2}{2}\norm{\Delta}^2
    \right\},
\]
whereas $\psi(W)=\E_Z[\varphi(W_t)].$ Substituting these identities into 
\eqref{eq:truncated-Stein} gives
\begin{equation}
\label{eq:shifted-Stein-identity}
    \E\!\left[
        \left\{
            W+R_t
            -
            G\one_{\mathcal A_t}(L_t-1)
        \right\}
        \cdot\varphi(W_t)
    \right]
    =0.
\end{equation}
Since 
$\E_Z[L_t]=1$,
\begin{align*} \E\!\left[ \norm{G}\one_{\mathcal A_t}|L_t-1| \right] \leq \E\!\left[ \norm{G}\one_{\mathcal A_t}(L_t+1) \right]= 2\E\!\left[ \norm{G}\one_{\mathcal A_t} \right] \leq 2\E[\norm{G}] <\infty. \end{align*}
Then, the random vector appearing inside the braces in
\eqref{eq:shifted-Stein-identity} is integrable. Since \eqref{eq:shifted-Stein-identity} holds for every bounded measurable
$\varphi$, it follows that
\begin{equation}
\label{eq:conditional-zero}
    \E\!\left[
        W+R_t
        -
        G\one_{\mathcal A_t}(L_t-1)
        \,\middle|\,
        W_t
    \right]
    =0.
\end{equation}
Combining \eqref{eq:relative-score} and \eqref{eq:conditional-zero}, we obtain
\begin{equation}
\label{eq:score-representation}
    \mathfrak{s}_t(W_t)
    =
    \alpha_t
    \E\!\left[
        -R_t-\lambda_t Z
        +
        G\one_{\mathcal A_t}(L_t-1)
        \,\middle|\,
        W_t
    \right].
\end{equation}
Since $W_t$ is measurable with respect to \(\sigma(W,Z)\), we proceed in two steps. We first show that $\E\left[
G\one_{\mathcal A_t}(L_t-1)
\middle|
W,Z
\right]$ admits an \(L_p\)-convergent Hermite expansion. We then apply the tower property together with conditional Jensen's inequality to bound the \(L_p\)-norm of the relative score.

\medskip
\noindent\textbf{Hermite expansion.}
Let \(H_k(z)\) denote the \(k\)th multivariate probabilists' Hermite tensor,
characterized by the generating identity
\begin{equation}
\label{eq:Hermite-generating}
    \exp\left\{
        \ip{u}{z}-\frac12\norm{u}^2
    \right\}
    =
    \sum_{k=0}^{\infty}
        \frac{1}{k!}
        u^{\ot k}\contr H_k(z),
    \qquad \forall u,z\in\R^d.
\end{equation}
We claim that
\begin{equation}
\label{eq:conditional-Hermite}
    \E\!\left[
        G\one_{\mathcal A_t}(L_t-1)
        \,\middle|\,
        W,Z
    \right]
    =
    \sum_{k=1}^{\infty}
        \frac{\lambda_t^k}{k!}
        T_{k,t}\contr H_k(Z)
    \qquad\text{in }L^p. 
\end{equation}

\begin{proof}[Proof of the Claim.]
We use the standard Hermite isometry
\begin{equation}
\label{eq:Hermite-isometry}
    \E_Z\!\left[
        \norm{B\contr H_k(Z)}^2
    \right]
    =
    k!\norm{B}_{\HS}^2,
    \qquad
    \forall B\in\R^d\ot\operatorname{Sym}^k(\R^d),
\end{equation} 
and the Gaussian hypercontractivity inequality Lemma 6 in 
\cite{bonis2020stein}: for every finite collection \(\{U_k(Z)\}_{k=0}^N\) of homogeneous Gaussian chaoses, where \(U_k\) has degree \(k\), it holds that
\begin{equation} \label{eq:Gaussian-hypercontractivity} \norm{\sum_{k=0}^N U_k(Z)}_{L_Z^p} \leq \left\{ \sum_{k=0}^N (p-1)^k \norm{U_k(Z)}_{L_Z^2}^2 \right\}^{1/2}. 
\end{equation} 
To prove the claim, define the partial sums
\[
    Q_N
    :=
    \one_{\mathcal A_t}
    \sum_{k=1}^N
        \frac{\lambda_t^k}{k!}
        \bigl(G\ot\Delta^{\ot k}\bigr)\contr H_k(Z).
\]
For $N>M$, applying \eqref{eq:Hermite-isometry} and \eqref{eq:Gaussian-hypercontractivity} gives
\begin{equation}
\label{eq:Hermite-tail}
    \norm{Q_N-Q_M}_{L_Z^p}^2
    \leq
    \norm{G}^2\one_{\mathcal A_t}
    \sum_{k=M+1}^N
        \frac{r_t^k}{k!},
    \qquad
    r_t:=(p-1)\lambda_t^2\norm{\Delta}^2.
\end{equation}
By \eqref{eq:small-increment}, $r_t\leq\frac1{16}$ on $\mathcal A_t.$ Since $G\in L^p$, \eqref{eq:Hermite-tail}  show
that $\{Q_N\}_{N \geq 0}$ is Cauchy in $L^p$. On the other hand, the identity
\eqref{eq:Hermite-generating}, applied with $u=\lambda_t\Delta$, gives 
\[
    Q_N
    \xlongrightarrow{N\rightarrow \infty}
    G\one_{\mathcal A_t}(L_t-1)\qquad\text{a.s.}
\]
Consequently,
\begin{equation}
\label{eq:Lt-Hermite-Lp}
    G\one_{\mathcal A_t}(L_t-1)
    =
    \sum_{k=1}^{\infty}
        \frac{\lambda_t^k}{k!}
        \bigl(G\ot\Delta^{\ot k}\bigr)\contr H_k(Z)
        \one_{\mathcal A_t}
    \qquad\text{in }L^p.
\end{equation}
Taking expectations condition on $(W,Z)$ and using the
independence of $Z$ from $(W,W',G)$  proves \eqref{eq:conditional-Hermite} of the claim. 
\end{proof}

For later use, note also that conditional Jensen's inequality and the
small-increment condition \eqref{eq:small-increment} imply
\begin{equation}
\label{eq:Tk-summability}
    \lambda_t^k
    \norm{T_{k,t}}_{L^p(\HS)}
    \leq
    \frac{1}{4^k(p-1)^{k/2}}
    \norm{G}_{L^p},
    \qquad k\geq1.
\end{equation}
In particular,
\[
    \sum_{k=1}^{\infty}
        \frac{(p-1)^k\lambda_t^{2k}}{k!}
        \norm{T_{k,t}}_{L^p(\HS)}^2
    <\infty.
\]

\medskip
\noindent\textbf{Control of the score.}
Note that $W_t$ is measurable with respect to $\sigma(W,Z)$. The tower property
in \eqref{eq:score-representation}, together with
\eqref{eq:conditional-Hermite} and conditional Jensen's inequality, gives 
\begin{align}
\label{eq:score-before-hypercontractivity}
    \norm{\mathfrak{s}_t(W_t)}_{L^p}
    \leq
    \alpha_t
    \Bigg\|
        -R_t
        +
        \lambda_t(T_{1,t}-I_d)\contr H_1(Z)
        +
        \sum_{k=2}^{\infty}
            \frac{\lambda_t^k}{k!}
            T_{k,t}\contr H_k(Z)
    \Bigg\|_{L^p},
\end{align}
where we used $H_1(Z)=Z$ and $I_d\contr H_1(Z)=Z.$ Conditioning on $W$, applying \eqref{eq:Gaussian-hypercontractivity} to finite
partial sums, and then passing to the infinite series, justified by
\eqref{eq:Tk-summability}, yields
\begin{align*}
    &
    \left\|
        \lambda_t(T_{1,t}-I_d)\contr H_1(Z)
        +
        \sum_{k=2}^{\infty}
            \frac{\lambda_t^k}{k!}
            T_{k,t}\contr H_k(Z)
    \right\|_{L^p}
    \\
    &\qquad\leq
    \sqrt{p-1}\,\lambda_t
    \norm{T_{1,t}-I_d}_{L^p(\HS)}
    +
    \left\{
        \sum_{k=2}^{\infty}
            \frac{(p-1)^k\lambda_t^{2k}}{k!}
            \norm{T_{k,t}}_{L^p(\HS)}^2
    \right\}^{1/2}.
\end{align*} 
Finally, applying the triangle inequality in
\eqref{eq:score-before-hypercontractivity}, and recalling that
$\alpha_t=e^{-t}$ and $\lambda_t=\delta_t^{-1}$, yields
\[
\norm{\mathfrak{s}_t(W_t)}_{L^p}
\leq
e^{-t}\Bigg[
    \norm{R_t}_{L^p}
    +
    \frac{\sqrt{p-1}}{\delta_t}
        \norm{T_{1,t}-I_d}_{L^p(\HS)}
    +
    \left\{
        \sum_{k=2}^{\infty}
        \frac{(p-1)^k}
             {k!\,\delta_t^{2k}}
        \norm{T_{k,t}}_{L^p(\HS)}^2
    \right\}^{1/2}
\Bigg],
\]
which is \eqref{eq:score-bound-general}.

\section{Refresh-Then-Maximal Coupling}
\label{sec:coupling}
In this section, we develop the coupling construction that will be 
employed with the relative score bound in Proposition~\ref{prop:OU-Stein} to prove our main result (Theorem \ref{thm:main}).
In particular, Lemma~\ref{lem:block-coupling} collects the key properties of our
\emph{refresh-then-maximal coupling}. We then describe the construction in detail
and verify that it yields the antisymmetric Stein coupling required by
Proposition~\ref{prop:OU-Stein}. The proof of
Lemma~\ref{lem:block-coupling} is deferred to
Section~\ref{sec:proof-block-coupling}.

Let $\Omega_+:=\mathsf X^{\mathbb N_0}$, and let \(\mathbf P_x\) denote the law on \(\Omega_+\) of a Markov chain with transition kernel \(P\) started from \(x\). We work with a two-sided stationary version \((X_k)_{k\in\mathbb Z}\) of the original chain and define
\[
    \mathcal X:=\sigma(X_k:k\in\mathbb Z),
    \qquad
    \mathcal P_j:=\sigma(X_k:k\leq j),
    \qquad j\in\mathbb Z.
\]

\begin{lemma}[Refresh-Then-Maximal Coupling]
\label{lem:block-coupling}
Suppose that \(\mathsf X\) is a standard Borel space and that \(P\) satisfies
\eqref{eq:UGE}. Then there exist an integer \(L\geq2\) and constants
\(C_{\mathrm c}<\infty\) and \(\rho_{\mathrm c}\in(0,1)\), depending only on
\((C_0,\rho_0)\), such that for each \(i\in\{1,\ldots,n\}\), on an extension of the original probability
space one can construct two suffixes
\[
    \omega^{i,r}
    :=
    (X_{i-1+k}^{i,r})_{k\geq0},
    \qquad r\in\{0,1\},
\]
with the following properties. 

\begin{enumerate}[label=\textup{(\roman*)}]
\item\textbf{Anchoring, conditional exchangeability, and refreshment.}
The first suffix is anchored to the observed chain,
\begin{equation}\label{eq:anchored-first-coordinate}
 X_j^{i,0}=X_j,
    \qquad j\geq i-1.
\end{equation}
Conditionally on \(\mathcal P_{i-1}\), the pair
\((\omega^{i,0},\omega^{i,1})\) is exchangeable and each \(\omega^{i,r}\) is a
\(P\)-chain started from \(X_{i-1}\). Moreover, for every bounded measurable
\(f:\mathsf X\to\mathbb R\),
\begin{equation}\label{eq:anchored-first-step-law}
\E\!\left[
        f(X_i^{i,1})
        \,\middle|\,
        \mathcal X
    \right]
    =
    Pf(X_{i-1})
    \qquad\text{a.s.}
\end{equation}
\item\textbf{Blockwise maximal coupling.}
Let $\mathcal J_m^i
    :=
    \sigma\!\left(
        X_{i-1+k}^{i,r}:
        0\leq k\leq m,\ r\in\{0,1\}
    \right).$ There exists an \((\mathcal J_m^i)_{m\geq0}\)-stopping time \(\kappa_i\) such
that
\begin{equation}
\label{eq:anchored-meet}
    X_{i-1+k}^{i,0}
    =
    X_{i-1+k}^{i,1},
    \qquad k\geq\kappa_i,
\end{equation}
and
\begin{equation}
\label{eq:anchored-tail}
    \Pp\!\left(
        \kappa_i>m
        \,\middle|\,
        \mathcal P_{i-1}
    \right)
    \leq
    C_{\mathrm c}\rho_{\mathrm c}^{\,m},
    \qquad m\geq0,
    \quad\text{a.s.}
\end{equation}
\item \textbf{Preservation of prefix measurability.}
Define
\[
    \mathfrak b(0):=0,
    \qquad
    \mathfrak b(m)
    :=
    1+L\left\lceil\frac{m-1}{L}\right\rceil,
    \qquad m\geq1,
\]
so that \(m\leq\mathfrak b(m)\leq m+L-1\).
Let \(H_i=H_i(\omega^{i,0},\omega^{i,1})\) be a finite-dimensional integrable function taking value in a Hilbert-space. Then $\E[H_i\mid\mathcal X]$ is measurable with respect to
\(\sigma(X_r:r\geq i-1)\). If, in addition, \(H_i\) is
\(\mathcal J_{\mathfrak b(m)}^i\)-measurable for some \(m\geq0\), then $\E[H_i\mid\mathcal X]$ is measurable with respect to $\sigma\!\left(
        X_{i-1},\ldots,X_{i-1+\mathfrak b(m)}
    \right).$
\end{enumerate}
\end{lemma}

\noindent\textbf{Comparison with standard couplings.}
The coupling in Lemma~\ref{lem:block-coupling} is designed to satisfy the
three properties above simultaneously. It combines an independent first-step
resampling with a symmetric blockwise maximal coupling. The independent
resampling provides the fresh first-step property needed to establish the
exact Stein identity \eqref{eq:Stein-identity}. The symmetry of the subsequent
coupling ensures conditional exchangeability of the two suffixes, while
maximal coupling yields a geometrically decaying
tail for the meeting time and hence sharp control of the coupling increment.

A further feature of the construction is that it preserves prefix
measurability, as formalized in
Lemma~\ref{lem:block-coupling}(iii). Specifically, if a functional depends
only on the coupled trajectories up to a given block boundary, then, after
conditioning on the observed trajectory, its conditional expectation depends
only on the observed trajectory up to the same boundary. This property allows
us to exploit the mixing behavior of the Markov chain when controlling the
fluctuations of the \textit{conditional} tensors arising in
Proposition~\ref{prop:OU-Stein}.

The preservation of prefix measurability follows from the recursive
blockwise construction. By contrast, classical maximal-coupling constructions for Markov chains may
use future trajectory information and therefore do not automatically preserve
the prefix-measurability property required here; see
\citet[Section~2]{moulos2021bicausal}. 
Consequently, a standard global maximal coupling cannot be applied directly
in our argument without additional structure.

\medskip
\noindent\textbf{Construction of the antisymmetric Stein coupling.}
Fix $i\in\{1,\ldots,n\}$. For $r\in\{0,1\}$ and $i\leq j\leq n$, set
\begin{equation}
\label{eq:coupled-weighted-increments}
    D_j^{i,r}:=g(X_j^{i,r})-Pg(X_{j-1}^{i,r}),
    \qquad
    Y_j^{i,r}:=M_{n,j}D_j^{i,r}.
\end{equation}
Each coupled coordinate has the same marginal law as the
stationary chain. Since the same deterministic matrix $M_{n,j}$
is used on both trajectories,
\begin{equation}
\label{eq:coupled-weighted-moments}
    \law(D_j^{i,r})=\law(D_j),
    \qquad
    \norm{Y_j^{i,r}}_{L^q}=a_j.
\end{equation}
Define
\[
    \Xi_i
    :=
    Y_i^{i,1}-Y_i^{i,0},
    \qquad
    \Delta_i
    :=
    \sum_{j=i}^n
        \bigl(
            Y_j^{i,1}-Y_j^{i,0}
        \bigr).
\]
We also introduce the two coupled versions of the full sum:
\[
    S_n^{i,r}
    :=
    \sum_{j=1}^{i-1}Y_j
    +
    \sum_{j=i}^nY_j^{i,r},
    \qquad r\in\{0,1\}.
\]
By \eqref{eq:anchored-first-coordinate},
\begin{equation}
\label{eq:coupled-sum-identities}
    S_n^{i,0}=S_n,
    \qquad
    S_n^{i,1}=S_n+\Delta_i,
    \qquad
    Y_i^{i,0}=Y_i.
\end{equation}
Let \(I\) be uniformly distributed on \(\{1,\ldots,n\}\), independently of
the original chain and all auxiliary coupling randomness. Define
\begin{equation}
\label{eq:Stein-triple}
    W:=S_n,
    \qquad
    W':=S_n^{I,1}=S_n+\Delta_I,
    \qquad
    G:=\frac n2\,\Xi_I,
    \qquad
    \Delta:=W'-W=\Delta_I.
\end{equation}
For \(t>0\), set 
\begin{equation*}
    \tau_t
    :=
    \frac{\delta_t}{4\sqrt p},
    \qquad
    \mathcal A_t
    :=
    \{\norm{\Delta}\leq\tau_t\}.
\end{equation*}

\medskip
\noindent\textbf{Verification of the conditions in Proposition~\ref{prop:OU-Stein}.} Because \(q=(2+\eta)p>p\) and all summations are over finitely many terms, 
\(W,W',G\in L^p(\R^d)\). Moreover, $\E[W]=0$ and $\E[WW^\top]=I_d,$ by the martingale-difference property and the covariance normalization assumption. We first verify that \((W,W',G)\) satisfies antisymmetry \eqref{eq:abstract-Stein-assumptions}. By
Lemma~\ref{lem:block-coupling}(i), conditionally on
\(\mathcal P_{i-1}\), the two coupled suffixes are exchangeable.
Since their common prefix
\(\sum_{j=1}^{i-1}Y_j\) is \(\mathcal P_{i-1}\)-measurable, we have $\bigl(S_n^{i,0},S_n^{i,1},\Xi_i\bigr)
    \stackrel{\mathrm d}{=}
    \bigl(S_n^{i,1},S_n^{i,0},-\Xi_i\bigr)$ for each $i\in \{1,\ldots,n\}$.  Therefore,
\[
    (W,W',G)
    \stackrel{\mathrm d}{=}
    (W',W,-G).
\]
We next verify the Stein identity~\eqref{eq:Stein-identity}.
The first-step freshness property~\eqref{eq:anchored-first-step-law}
and integrable truncation imply
\begin{align}
\label{eq:poisson-refresh-centering}
    \E[Y_i^{i,1}\mid\mathcal X]
    &=M_{n,i}\left\{
        \E[g(X_i^{i,1})\mid\mathcal X]-Pg(X_{i-1})
      \right\}\nonumber\\
    &=M_{n,i}\bigl\{Pg(X_{i-1})-Pg(X_{i-1})\bigr\}=0
      \qquad\text{a.s.}
\end{align}
Indeed, $g(X_i^{i,1})$ is integrable by its stationary
marginal law, and the conditional law of $X_i^{i,1}$ given
$\mathcal X$ is $P(X_{i-1},\cdot)$.
Thus the centering needed by the Stein identity is automatic
for the weighted Poisson martingale differences.

Because \(S_n^{i,0}=S_n\) is \(\mathcal X\)-measurable, for every bounded
measurable \(\varphi:\R^d\to\R^d\), we have $ \E\!\left[
        Y_i^{i,1}\cdot\varphi(S_n^{i,0})
    \right]=0.$  Conditional exchangeability of \((\omega^{i,0},\omega^{i,1})\) by Lemma~\ref{lem:block-coupling}(i) further implies
\[
    \E\!\left[
        Y_i^{i,0}\cdot\varphi(S_n^{i,1})
    \right]=0,
    \qquad \text{and} \qquad
    \E\!\left[
        Y_i^{i,1}\cdot\varphi(S_n^{i,1})
    \right]
    =
    \E\!\left[
        Y_i^{i,0}\cdot\varphi(S_n^{i,0})
    \right].
\]
Using \(\Xi_i=Y_i^{i,1}-Y_i^{i,0}\) and
\eqref{eq:coupled-sum-identities}, we obtain
\begin{equation}
\label{eq:fixed-i-Stein}
    \E\!\left[
        \Xi_i\cdot
        \bigl\{
            \varphi(S_n^{i,1})-\varphi(S_n^{i,0})
        \bigr\}
    \right]
    =
    2\E\!\left[
        Y_i\cdot\varphi(S_n)
    \right].
\end{equation}
Averaging over \(I\) gives
\[
\begin{aligned}
    \E\!\left[
        G\cdot\{\varphi(W')-\varphi(W)\}
    \right]
    &=
    \frac12\sum_{i=1}^n
    \E\!\left[
        \Xi_i\cdot
        \{\varphi(S_n^{i,1})-\varphi(S_n^{i,0})\}
    \right] \\
    &=
    \sum_{i=1}^n
    \E\!\left[
        Y_i\cdot\varphi(S_n)
    \right]
    =
    \E\!\left[
        W\cdot\varphi(W)
    \right],
\end{aligned}
\]
which is \eqref{eq:Stein-identity}.  Finally, under
\((W,W',G)\mapsto(W',W,-G)\), the increment
\(\Delta=W'-W\) is mapped to \(-\Delta\). Hence
\(\mathcal A_t=\{\norm{\Delta}\leq\tau_t\}\) satisfies \eqref{eq:At_invariance}, and under
\(\mathcal A_t\), it holds that
\[
    \frac{\sqrt{p-1}}{\delta_t}\norm{\Delta}
    \leq
    \frac{\sqrt{p-1}}{4\sqrt p}
    \leq
    \frac14.
\]
Consequently, we can apply Proposition \ref{prop:OU-Stein} to our fresh-then-maximal coupling to bound the relative score \(\norm{\mathfrak{s}_t(W_t)}_{L^p}\). By bounding the three terms in operator discrepancies separately, as stated in Lemma~\ref{lem:tensors}, and applying Lemma~\ref{lem:OU-transport}, we then derive the $\W_p$ Gaussian approximation of $S_n$ in Theorem~\ref{thm:main}. The detailed proof is provided in Section~\ref{sec:completion}.
 
\subsection{Proof of Lemma~\ref{lem:block-coupling}}
\label{sec:proof-block-coupling}

We will consider blockwise maximal coupling of block length $L$. 
Choose \(L\geq2\) sufficiently large that
\begin{equation}
\label{eq:block-Dobrushin}
    \sup_{u,v\in\mathsf X}
    \norm{P^L(u,\cdot)-P^L(v,\cdot)}_{\TV}
    \leq
    2C_0\rho_0^L
    \leq
    \frac12.
\end{equation}
We first record the measurable coupling ingredients used below. Since
\(\mathsf X\) is a standard Borel space, measurable maximal-coupling and
disintegration results allow the following kernels to be chosen jointly
measurably; see, e.g.,
\cite{kallenberg1997foundations}.
For each \(u,v\in\mathsf X\), let
\(\Gamma_{u,v}\) be a symmetric maximal coupling of
\(P^L(u,\cdot)\) and \(P^L(v,\cdot)\), satisfying
\begin{equation}
\label{eq:symmetric-maximal-coupling}
    \Gamma_{u,v}
    =
    \mathsf s_{\#}\Gamma_{v,u},
    \qquad
    \mathsf s(y,z):=(z,y).
\end{equation}
Here, \(\mathsf s_{\#}\Gamma_{v,u}\) denotes the pushforward of
\(\Gamma_{v,u}\) under \(\mathsf s\); that is, for every measurable
\(A\subseteq\mathsf X\times\mathsf X\),
\[
    \mathsf s_{\#}\Gamma_{v,u}(A)
    :=
    \Gamma_{v,u}\!\left(\mathsf s^{-1}(A)\right).
\]
Such symmetry may be enforced, starting from any measurable family
\(\{\Gamma_{u,v}^0\}\) of maximal couplings, by replacing it with
\[
    \Gamma_{u,v}
    :=
    \frac12
    \left(
        \Gamma_{u,v}^0
        +
        \mathsf s_{\#}\Gamma_{v,u}^0
    \right).
\]
Such coupling preserves both the marginals and maximality. Let
\(\mathsf B_L(u,y;\cdot)\) denote a measurable \(L\)-step Markov-bridge
kernel, so that the law of an \(L\)-step segment of a \(P\)-chain started
from \(u\) admits the disintegration
\begin{align}
\label{eq:bridge-disintegration}
 P(u,\mathrm dx_1)
  P(x_1,\mathrm dx_2)
  \cdots
  P(x_{L-1},\mathrm dy)=P^L(u,\mathrm dy)\,
\mathsf B_L
\bigl(
    u,y;
    \mathrm dx_1\cdots\mathrm dx_{L-1}
\bigr).
\end{align}
Finally, there exists a measurable probability kernel $\Lambda$ such that 
\begin{equation}
\label{eq:maximal-coupling-disintegration}
    \Gamma_{u,v}(\mathrm dy,\mathrm dz)
    =
    P^L(u,\mathrm dy)\,
    \Lambda(u,v,y;\mathrm dz).
\end{equation}

\noindent\textbf{Conditional construction given the anchored trajectory.}
Fix \(x\in\mathsf X\) and a realization
\[
    \omega^0=(\omega_k^0)_{k\geq0},
    \qquad
    \omega_0^0=x,
\]
of the first trajectory. We construct a second trajectory
\(\omega^1=(\omega_k^1)_{k\geq0}\) conditionally on \(\omega^0\).

Set \(\omega_0^1=x\), and sample $\omega_1^1\sim P(x,\cdot)$ independently of the entire trajectory \(\omega^0\). If
\(\omega_1^1=\omega_1^0\), set
\(\omega_k^1=\omega_k^0\) for all \(k\geq1\). Otherwise, introduce the block boundaries
\[
    t_j:=1+jL,
    \qquad j\geq0.
\]
Suppose that the two trajectories have not yet coupled at time \(t_j\), and
write
\[
    u:=\omega_{t_j}^0,
    \qquad
    v:=\omega_{t_j}^1,
    \qquad
    y:=\omega_{t_{j+1}}^0.
\]
Sample the
endpoint \(z:=\omega_{t_{j+1}}^1\) of the second trajectory according to $z\sim\Lambda(u,v,y;\cdot),$ and then sample the intermediate states of the second trajectory according to
the bridge law $\mathsf B_L(v,z;\cdot).$ If \(z=y\), couple the two trajectories identically from time \(t_{j+1}\)
onward; otherwise, repeat the same procedure on the next block.

This recursive construction defines a measurable probability  kernel
\[
    \mathbf K(x,\omega^0; \mathrm d\omega^1).
\]
This kernel has two important properties. First, because the
first transition is sampled independently of the entire anchored trajectory,
for every bounded measurable \(f\),
\begin{equation}
\label{eq:kernel-first-step-freshness}
    \int
        f(\omega_1^1)\,
        \mathbf K(x,\omega^0; \mathrm d\omega^1)
    =
    Pf(x).
\end{equation}
Second, for every block boundary $b\in\{0,1,1+L,1+2L,\ldots\},$ the conditional law of the second trajectory $\omega^1$ through time \(b\) depends on
the anchored trajectory $\omega^0$ only through the prefix $(\omega_0^0,\ldots,\omega_b^0)$. Equivalently, there
exists a measurable probability kernel \(\mathbf K_b\) such that
\begin{equation}
\label{eq:causal-prefix-kernel}
\begin{aligned}
\mathbf K\!\left(
    x,\omega^0;
    \{\omega^1:
        (\omega_0^1,\ldots,\omega_b^1)\in A
    \}
\right)=
\mathbf K_b\!\left(
    x,
    (\omega_0^0,\ldots,\omega_b^0);
    A
\right)
\end{aligned}
\end{equation}
for every measurable \(A\subseteq\mathsf X^{b+1}\).

\medskip
\noindent\textbf{Symmetry and marginal laws.}
Define the joint law of the two trajectories by
\[
    \mathbf Q_x(\mathrm d\omega^0,\mathrm d\omega^1)
    :=
    \mathbf P_x(\mathrm d\omega^0)\,
    \mathbf K(x,\omega^0;\mathrm d\omega^1).
\]
At the first step, \(\omega_1^0\) and \(\omega_1^1\) are sampled
independently from the common distribution \(P(x,\cdot)\), and hence their
joint law is invariant under interchanging the two trajectories.

Now consider a block before the two trajectories meet, and let
\((u,v)\) be their states at the beginning of the block. Denote by
\((y,z)\) their states at the end of the block, and by
\((\mathbf x,\mathbf x')\) their respective interior states. By
\eqref{eq:bridge-disintegration} and
\eqref{eq:maximal-coupling-disintegration}, the conditional joint law of
the two blocks is
\begin{equation}
\label{eq:block_law}
    \Gamma_{u,v}(\mathrm dy,\mathrm dz)\,
    \mathsf B_L(u,y;\mathrm d\mathbf x)\,
    \mathsf B_L(v,z;\mathrm d\mathbf x').
\end{equation}
Since
\(\Gamma_{u,v}=\mathsf s_{\#}\Gamma_{v,u}\), interchanging the two
trajectories transforms \eqref{eq:block_law} into the same block law with
the initial states interchanged from \((u,v)\) to \((v,u)\). Thus, the
block transition rule is invariant under swapping the two trajectories.
The diagonal coupling used after the trajectories meet clearly has the
same property. Therefore, by induction over successive blocks, every finite pair of
trajectory prefixes is exchangeable. It follows that
\begin{equation}
\label{eq:Q-symmetry}
    (\omega^0,\omega^1)
    \stackrel{\mathrm d}{=}
    (\omega^1,\omega^0)
    \qquad\text{under }\mathbf Q_x.
\end{equation}
Finally, the first marginal of \(\mathbf Q_x\) is \(\mathbf P_x\) by
definition, and hence exchangeability implies that the second marginal is
also \(\mathbf P_x\).

\medskip
\noindent\textbf{Geometric meeting time.} With a slight abuse of notation, we also use \(\mathbf Q_x\) to denote
probabilities and conditional probabilities under this joint law. Let
\[
    N
    :=
    \inf\{j\geq0:
        \omega_{t_j}^0=\omega_{t_j}^1\},
    \qquad
    \kappa:=1+LN.
\]
At any block boundary at which the trajectories have not yet met, since \(\Gamma_{u,v}\) is a maximal coupling, \eqref{eq:block-Dobrushin} implies
\[
\begin{aligned}
&\mathbf Q_x\!\left(
    \omega_{t_{j+1}}^0=\omega_{t_{j+1}}^1
    \,\middle|\,
    \sigma(\omega_k^0,\omega_k^1:0\leq k\leq t_j)
\right)
\\
&\qquad=
1-\norm{P^L(u,\cdot)-P^L(v,\cdot)}_{\TV}
\geq
\frac12.
\end{aligned}
\]
Therefore
there exist \(C_{\mathrm c}<\infty\) and
\(\rho_{\mathrm c}\in(0,1)\), depending only on \(L\), such that
\begin{equation}
\label{eq:canonical-meeting-tail}
    \sup_{x\in\mathsf X}
    \mathbf Q_x(\kappa>m)
    \leq
    C_{\mathrm c}\rho_{\mathrm c}^{\,m},
    \qquad m\geq0.
\end{equation}
Since coupling is checked at block boundaries and is permanent once it
occurs, \(\kappa\) is a stopping time for the joint prefix filtration.

\medskip
\noindent\textbf{Transfer to the stationary chain.}
For each \(i\in\{1,\ldots,n\}\), set $\omega^{i,0}
    :=
    (X_{i-1+k})_{k\geq0},$ and, conditionally on \(\mathcal X\), generate $\omega^{i,1}
    \sim
    \mathbf K
    (X_{i-1},\omega^{i,0},\cdot).$ The anchoring property
\eqref{eq:anchored-first-coordinate} is immediate. Conditionally on \(\mathcal P_{i-1}\), the Markov property implies that the
observed suffix \(\omega^{i,0}\) has law
\(\mathbf P_{X_{i-1}}\). Hence, by the definition of
\(\mathbf Q_x\),
\[
    \law\!\left(
        \omega^{i,0},\omega^{i,1}
        \,\middle|\,
        \mathcal P_{i-1}
    \right)
    =
    \mathbf Q_{X_{i-1}}.
\]
It follows from \eqref{eq:Q-symmetry} that
\((\omega^{i,0},\omega^{i,1})\) is conditionally exchangeable given
\(\mathcal P_{i-1}\), and both coordinates are \(P\)-chains started from
\(X_{i-1}\). This proves the first part of
Lemma~\ref{lem:block-coupling}(i). The freshness property follows directly from
\eqref{eq:kernel-first-step-freshness}. Indeed, for every bounded measurable
\(f\),
\begin{align*}
    \E\!\left[
        f(X_i^{i,1})
        \,\middle|\,
        \mathcal X
    \right]
    =
    \int
        f(\omega_1^1)\,
        \mathbf K
        (X_{i-1},\omega^{i,0},\mathrm d\omega^1)=
    Pf(X_{i-1}),\quad\text{a.s.}
\end{align*}
Let \(\kappa_i\) be the meeting time obtained from the canonical construction
applied to
\((\omega^{i,0},\omega^{i,1})\). Then
\eqref{eq:anchored-meet} holds by construction. Moreover,
\eqref{eq:canonical-meeting-tail} and the preceding conditional-law identity
give
\[
    \Pp\!\left(
        \kappa_i>m
        \,\middle|\,
        \mathcal P_{i-1}
    \right)
    =
    \mathbf Q_{X_{i-1}}(\kappa >m)
    \leq
    C_{\mathrm c}\rho_{\mathrm c}^{\,m},
    \qquad m\geq0,
    \quad\text{a.s.},
\]
which proves \eqref{eq:anchored-tail} and
Lemma~\ref{lem:block-coupling}(ii).  

Finally, let
\(H_i=H_i(\omega^{i,0},\omega^{i,1})\) be integrable. By the conditional
construction,
\begin{equation}
\label{eq:conditional-kernel-proof}
    \E[H_i\mid\mathcal X]
    =
    \int
        H_i(\omega^{i,0},\omega^1)\,
        \mathbf K
        (X_{i-1},\omega^{i,0},\mathrm d\omega^1).
\end{equation}
The right-hand side depends on the observed trajectory only through the
suffix \((X_r)_{r\geq i-1}\), which proves the first assertion in
Lemma~\ref{lem:block-coupling}(iii). Now suppose that \(H_i\) is
\(\mathcal J_{\mathfrak b(m)}^i\)-measurable. By definition, \(\mathfrak b(m)\) is a block boundary. \eqref{eq:causal-prefix-kernel} implies that the
conditional law of $(X_{i-1+k}^{i,1})_{0\leq k\leq \mathfrak b(m)}$ given the observed trajectory depends on that trajectory only through $X_{i-1},\ldots,X_{i-1+\mathfrak b(m)}.$ Since \(H_i\) itself depends only on the two coupled prefixes through time
\(\mathfrak b(m)\), \eqref{eq:conditional-kernel-proof} is therefore measurable with
respect to $\sigma\!\left(
        X_{i-1},\ldots,X_{i-1+\mathfrak b(m)}
    \right).$ This proves Lemma~\ref{lem:block-coupling}(iii) and completes the proof.

\section{Completion of the Proof of Theorem~\ref{thm:main}}
\label{sec:completion}
The weighted Poisson representation~\eqref{eq:coupled-weighted-increments}
has the localization property needed below. If $\kappa_i\leq\ell$,
then the two suffixes agree at both relative times $\ell$ and
$\ell+1$, that is, at times $i+\ell-1$ and $i+\ell$. Hence
\begin{equation}
\label{eq:poisson-increment-localization}
    Y_{i+\ell}^{i,1}-Y_{i+\ell}^{i,0}
    =\bigl(Y_{i+\ell}^{i,1}-Y_{i+\ell}^{i,0}\bigr)
       \one_{\{\kappa_i>\ell\}},
    \qquad 0\leq\ell\leq n-i.
\end{equation}
The indicator is $\{\kappa_i>\ell\}$ rather than
$\{\kappa_i>\ell+1\}$, because a Poisson martingale
increment depends on both endpoints. Together with~\eqref{eq:coupled-weighted-moments} and the
refresh identity~\eqref{eq:poisson-refresh-centering}, this supplies
the increment properties used in the following estimates.

Define
\[
    B_p
    :=
    p^{5/2}\norm{A}_4^2,
\]
and
\[
    C_1
    :=
    \sum_{i=1}^n
        \E\!\left[
            \norm{\Xi_i}\norm{\Delta_i}
        \right],
    \qquad
    C_3
    :=
    \sum_{i=1}^n
        \E\!\left[
            \norm{\Xi_i}\norm{\Delta_i}^3
        \right].
\]
Recall that
\(\mathcal A_t=\{\norm{\Delta}\leq\tau_t\}\).
Conditioning on the uniform index \(I\) in \eqref{eq:Stein-triple} gives
\begin{align*}
    R_t
    &=
    \sum_{i=1}^n
        \E\!\left[
            \Xi_i
            \one_{\{\norm{\Delta_i}>\tau_t\}}
            \,\middle|\,
            W
        \right],
    \\
    T_{k,t}
    &=
    \frac12
    \sum_{i=1}^n
        \E\!\left[
            \Xi_i\ot\Delta_i^{\ot k}
            \one_{\{\norm{\Delta_i}\leq\tau_t\}}
            \,\middle|\,
            W
        \right],
    \qquad k\geq1.
\end{align*}
For \(\tau>0\), define
\begin{align*}
    \mathfrak m_1(\tau)
    &:=
    \sum_{i=1}^n
        \E\!\left[
            \norm{\Xi_i}\norm{\Delta_i}
            \one_{\{\norm{\Delta_i}>\tau\}}
        \right],
    \\
    \mathfrak m_3(\tau)
    &:=
    \sum_{i=1}^n
        \E\!\left[
            \norm{\Xi_i}\norm{\Delta_i}^3
            \one_{\{\norm{\Delta_i}\leq\tau\}}
        \right].
\end{align*}
The estimates needed
below are collected in the following lemma.

\begin{lemma}[Tensor Estimates]
\label{lem:tensors}
For every \(t>0\),
\begin{align}
    \norm{R_t}_{L^p}
    &\leq
    C(\eta,C_0,\rho_0)\tau_t^{-1}B_p,
    \label{eq:R-bound}
    \\
    \norm{T_{1,t}-I_d}_{L^p(\HS)}
    &\leq
    C(\eta,C_0,\rho_0)\left\{
        B_p+\mathfrak m_1(\tau_t)
    \right\},
    \label{eq:T1-bound}
    \\
    \norm{T_{k,t}}_{L^p(\HS)}
    &\leq
    C(\eta,C_0,\rho_0)\left\{
        \tau_t^{k-1}B_p
        +
        \one_{\{k\geq3,\;k\ \mathrm{odd}\}}
        \tau_t^{k-3}\mathfrak m_3(\tau_t)
    \right\},
    \qquad k\geq2.
    \label{eq:Tk-bound}
\end{align}
Moreover, for every \(\tau>0\),
\begin{align}
    \mathfrak m_1(\tau)
    &\leq
    \min\left\{
        C_1,\tau^{-2}C_3
    \right\},
    \qquad
    \mathfrak m_3(\tau)
    \leq
    \min\left\{
        C_3,\tau^2C_1
    \right\},
    \label{eq:M-min}
    \\
    C_1
    &\leq
    C(\eta,C_0,\rho_0)\sqrt d\,\norm{A}_2,
    \qquad
    C_3
    \leq
    C(\eta,C_0,\rho_0)\norm{A}_4^4.
    \label{eq:C1-C3-bounds}
\end{align}
\end{lemma}

Since $\frac{(p-1)\tau_t^2}{\delta_t^2}
    =
    \frac{p-1}{16p}
    \leq
    \frac1{16},$ Lemma~\ref{lem:tensors} and Minkowski's inequality in the weighted
\(\ell_2\)-sum over \(k\) give
\begin{align*}
    &\tau_t^{-1}B_p
    +
    \frac{\sqrt{p-1}}{\delta_t}B_p
    +
    \left\{
        \sum_{k=2}^\infty
        \frac{(p-1)^k\tau_t^{2k-2}}
             {k!\,\delta_t^{2k}}
        B_p^2
    \right\}^{1/2}
    \leq
    C(\eta,C_0,\rho_0)\frac{\sqrt p\,B_p}{\delta_t},
    \\
    &\left\{
        \sum_{\substack{k\geq3\\k\ \mathrm{odd}}}
        \frac{(p-1)^k\tau_t^{2k-6}}
             {k!\,\delta_t^{2k}}
        \mathfrak m_3(\tau_t)^2
    \right\}^{1/2}
    \leq
    C(\eta,C_0,\rho_0)\frac{p^{3/2}}{\delta_t^3}
        \mathfrak m_3(\tau_t).
\end{align*}
Consequently, Proposition~\ref{prop:OU-Stein} yields, for every \(t>0\),
\begin{equation}
\label{eq:score-before-min}
    \norm{\mathfrak{s}_t(W_t)}_{L^p}
    \leq
    C(\eta,C_0,\rho_0)e^{-t}\left[
        \frac{\sqrt p\,B_p}{\delta_t}
        +
        \frac{\sqrt p}{\delta_t}\mathfrak m_1(\tau_t)
        +
        \frac{p^{3/2}}{\delta_t^3}\mathfrak m_3(\tau_t)
    \right].
\end{equation}
To control the last two terms, set
\[
    a_t:=\frac{\sqrt p\,C_1}{\delta_t},
    \qquad
    b_t:=\frac{p^{3/2}C_3}{\delta_t^3}.
\]
Using \eqref{eq:M-min} and
\(\tau_t^2=\delta_t^2/(16p)\), we obtain
\[
    \frac{\sqrt p}{\delta_t}\mathfrak m_1(\tau_t)
    \leq
    \min\{a_t,16b_t\},
    \qquad
    \frac{p^{3/2}}{\delta_t^3}\mathfrak m_3(\tau_t)
    \leq
    \min\left\{b_t,\frac{a_t}{16}\right\}.
\]
Hence
\[
    \frac{\sqrt p}{\delta_t}\mathfrak m_1(\tau_t)
    +
    \frac{p^{3/2}}{\delta_t^3}\mathfrak m_3(\tau_t)
    \leq
    17\min\{a_t,b_t\}.
\]
Since
\[
    \sqrt p\,B_p
    =
    p^3\norm{A}_4^2,
\]
substitution into \eqref{eq:score-before-min} gives
\begin{equation}
\label{eq:score-final}
    \norm{\mathfrak{s}_t(W_t)}_{L^p}
    \leq
    C(\eta,C_0,\rho_0)e^{-t}\Bigg[
        \frac{
            p^3\norm{A}_4^2
        }{\delta_t}
        +
        \min\left\{
            \frac{\sqrt p\,C_1}{\delta_t},
            \frac{p^{3/2}C_3}{\delta_t^3}
        \right\}
    \Bigg].
\end{equation}
It remains to integrate the pointwise score estimate. Under the change of
variables $u=\delta_t=\sqrt{e^{2t}-1},$ we have
\[
    e^{-t}=(1+u^2)^{-1/2},
    \qquad
    \mathrm dt=\frac{u}{1+u^2}\,\mathrm du.
\]
Therefore,
\begin{equation}
\label{eq:first-integral}
    \int_0^\infty
        \frac{e^{-t}}{\delta_t}\,\mathrm dt
    =
    \int_0^\infty
        \frac{\mathrm du}{(1+u^2)^{3/2}}
    =
    1.
\end{equation}
Moreover, for every \(a,b\geq0\),
\begin{equation}
\label{eq:min-integral}
    \int_0^\infty
        e^{-t}
        \min\left\{
            \frac{a}{\delta_t},
            \frac{b}{\delta_t^3}
        \right\}
        \,\mathrm dt
    =
    \int_0^\infty
        \frac{\min\{a,b/u^2\}}
             {(1+u^2)^{3/2}}
        \,\mathrm du
    \leq
    2\sqrt{ab}.
\end{equation}
Indeed, when \(ab>0\), the last inequality follows by discarding the
denominator and splitting the integral at \(u=\sqrt{b/a}\); the case
\(ab=0\) is immediate. Applying Lemma~\ref{lem:OU-transport} to \eqref{eq:score-final}, together
with \eqref{eq:first-integral}--\eqref{eq:min-integral} with $a=\sqrt p C_1$ and $ b=p^{3/2}C_3,$ gives
\begin{equation}
\label{eq:before-low-moments}
    \W_p\!\left(
        \law(W),
        \mathcal N(0,I_d)
    \right)
    \leq
    C(\eta,C_0,\rho_0)\left\{
        p^3\norm{A}_4^2
        +
        p\sqrt{C_1C_3}
    \right\}.
\end{equation}
Finally, \eqref{eq:C1-C3-bounds} implies
\[
    \sqrt{C_1C_3}
    \leq
    C(\eta,C_0,\rho_0) d^{1/4}
    \norm{A}_2^{1/2}
    \norm{A}_4^2.
\]
Since \(W=S_n\), substituting this bound into
\eqref{eq:before-low-moments} proves \eqref{eq:main-bound} and completes the
proof of Theorem~\ref{thm:main}.

\subsection{Proof of Lemma~\ref{lem:tensors}}
\label{sec:proof-tensors}

We first isolate the conditional-fluctuation estimate. Its proof is deferred to
Section~\ref{sec:proof-suffix-fluctuation}.

\begin{lemma}
\label{lem:suffix-fluctuation}
Let \(\mathsf H\) be a finite-dimensional real Hilbert space. For each
\(1\leq i\leq n\), let
\(\phi_i:\R^d\times\R^d\to\mathsf H\) be measurable and set $\Phi_i:=\phi_i(\Xi_i,\Delta_i).$ Suppose that, for some deterministic \(\Lambda\geq0\),
\begin{equation}
\label{eq:suffix-functional-domination}
    \norm{\Phi_i}_{\mathsf H}
    \leq
    \Lambda\norm{\Xi_i}\norm{\Delta_i}
    \qquad\text{a.s.}
\end{equation}
Then
\begin{equation}
\label{eq:suffix-fluctuation}
    \norm{
        \sum_{i=1}^n
        \bigl(
            \E[\Phi_i\mid W]-\E[\Phi_i]
        \bigr)
    }_{L^p(\mathsf H)}
    \leq
    C(\eta,C_0,\rho_0)
    \Lambda B_p.
\end{equation}
\end{lemma}

By the conditional
exchangeability in Lemma~\ref{lem:block-coupling}(i),
\begin{equation}
\label{eq:tensor-central-symmetry}
    (\Xi_i,\Delta_i)
    \stackrel{\mathrm d}{=}
    (-\Xi_i,-\Delta_i),
    \qquad
    1\leq i\leq n.
\end{equation}
Apply Lemma~\ref{lem:suffix-fluctuation} to
\[
    \Phi_i
    :=
    \Xi_i
    \one_{\{\norm{\Delta_i}>\tau_t\}}.
\]
By \eqref{eq:tensor-central-symmetry},
\(\E[\Phi_i]=0\), while $\norm{\Phi_i}
    \leq
    \tau_t^{-1}\norm{\Xi_i}\norm{\Delta_i}.$ Since $R_t
    =
    \sum_{i=1}^n\E[\Phi_i\mid W],$ Lemma~\ref{lem:suffix-fluctuation} gives
\[
    \norm{R_t}_{L^p}
    \leq
    C(\eta,C_0,\rho_0)\tau_t^{-1}B_p,
\]
which proves \eqref{eq:R-bound}. Recall from \eqref{eq:Stein-covariance-calibration} that
\[
    \frac12
    \sum_{i=1}^n
        \E[\Xi_i\ot\Delta_i]
    =
    \E[G\ot\Delta]
    =
    I_d.
\]
Therefore,
\begin{align*}
    T_{1,t}-I_d
    ={}&
    \frac12\sum_{i=1}^n
    \Bigl(
        \E[\Xi_i\ot\Delta_i\mid W]
        -
        \E[\Xi_i\ot\Delta_i]
    \Bigr)
    \\
    &-
    \frac12\sum_{i=1}^n
    \Bigl(
        \E[
            \Xi_i\ot\Delta_i
            \one_{\{\norm{\Delta_i}>\tau_t\}}
            \mid W
        ]
        -
        \E[
            \Xi_i\ot\Delta_i
            \one_{\{\norm{\Delta_i}>\tau_t\}}
        ]
    \Bigr)
    \\
    &-
    \frac12\sum_{i=1}^n
        \E\!\left[
            \Xi_i\ot\Delta_i
            \one_{\{\norm{\Delta_i}>\tau_t\}}
        \right].
\end{align*}
The first two sums are bounded in \(L^p(\HS)\) by
\(C(\eta,C_0,\rho_0)B_p\) through Lemma~\ref{lem:suffix-fluctuation}, since the
Hilbert--Schmidt norm of each corresponding functional is bounded by
\(\norm{\Xi_i}\norm{\Delta_i}\). The last sum has Hilbert--Schmidt norm at
most \(\frac12\mathfrak m_1(\tau_t)\). Hence
\[
    \norm{T_{1,t}-I_d}_{L^p(\HS)}
    \leq
    C(\eta,C_0,\rho_0)\left\{
        B_p+\mathfrak m_1(\tau_t)
    \right\},
\]
which proves \eqref{eq:T1-bound}. For \(k\geq2\), apply Lemma~\ref{lem:suffix-fluctuation} to
\[
    \Phi_i
    :=
    \Xi_i\ot\Delta_i^{\ot k}
    \one_{\{\norm{\Delta_i}\leq\tau_t\}}.
\]
Since the Hilbert--Schmidt norm is multiplicative on rank-one tensors,
\[
    \norm{\Phi_i}_{\HS}
    \leq
    \tau_t^{k-1}
    \norm{\Xi_i}\norm{\Delta_i}.
\]
Therefore,
\begin{equation}
\label{eq:Tk-centered}
    \norm{T_{k,t}-\E [T_{k,t}]}_{L^p(\HS)}
    \leq
    C(\eta,C_0,\rho_0)\tau_t^{k-1}B_p.
\end{equation}
Under \eqref{eq:tensor-central-symmetry}, \(\Phi_i\) is mapped to
\((-1)^{k+1}\Phi_i\). Consequently,
\[
    \E[T_{k,t}]=0
    \qquad\text{when \(k\) is even}.
\]
For odd \(k\geq3\),
\[
    \norm{\E [T_{k,t}]}_{\HS}
    \leq
    \frac12\sum_{i=1}^n
        \E\!\left[
            \norm{\Xi_i}
            \norm{\Delta_i}^k
            \one_{\{\norm{\Delta_i}\leq\tau_t\}}
        \right]\leq
    \frac12
    \tau_t^{k-3}
    \mathfrak m_3(\tau_t).
\]
Combining this bound with \eqref{eq:Tk-centered} proves
\eqref{eq:Tk-bound}. For every \(\tau>0\),
\[
    \one_{\{\norm{\Delta_i}>\tau\}}
    \leq
    \tau^{-2}\norm{\Delta_i}^2,
    \qquad
    \norm{\Delta_i}^2
    \one_{\{\norm{\Delta_i}\leq\tau\}}
    \leq
    \tau^2.
\]
Multiplying by
\(\norm{\Xi_i}\norm{\Delta_i}\), summing over \(i\), and combining with
the trivial bounds yields
\[
    \mathfrak m_1(\tau)
    \leq
    \min\{C_1,\tau^{-2}C_3\},
    \qquad
    \mathfrak m_3(\tau)
    \leq
    \min\{C_3,\tau^2C_1\},
\]
which proves \eqref{eq:M-min}. It remains to control \(C_1\) and \(C_3\). For \(r\in\{2,4\}\), Lemma~\ref{lem:block-coupling}(ii) gives
\[
    Y_{i+\ell}^{i,1}-Y_{i+\ell}^{i,0}
    =
    \bigl(
        Y_{i+\ell}^{i,1}-Y_{i+\ell}^{i,0}
    \bigr)
    \one_{\{\kappa_i>\ell\}}.
\]
By Lemma~\ref{lem:block-coupling}(i), both coupled coordinates have the
same marginal law as the stationary chain. Hence H\"older's inequality and
the geometric tail \eqref{eq:anchored-tail} imply
\[
    \norm{
        Y_{i+\ell}^{i,1}-Y_{i+\ell}^{i,0}
    }_{L^r}
    \leq
    2a_{i+\ell}
    \Pp(\kappa_i>\ell)^{1/r-1/q}
    \leq
    C(\eta,C_0,\rho_0)\rho^\ell a_{i+\ell},
\]
for some \(\rho\in(0,1)\) depending only on
\((\eta,C_0,\rho_0)\). Minkowski's inequality and Young's convolution inequality therefore give
\begin{equation}
\label{eq:Delta-low-moment}
    \norm{\Delta_i}_{L^r}
    \leq
    C(\eta,C_0,\rho_0)\sum_{\ell=0}^{n-i}
        \rho^\ell a_{i+\ell},
    \qquad
    \left\|
        \left(
            \sum_{\ell=0}^{n-i}
                \rho^\ell a_{i+\ell}
        \right)_{i=1}^n
    \right\|_r
    \leq
    C(\eta,C_0,\rho_0)\norm{A}_r.
\end{equation}
Since $\norm{\Xi_i}_{L^2}
    \leq
    2\norm{Y_i}_{L^2},$ Cauchy--Schwarz, martingale orthogonality, and
\(\operatorname{Cov}(S_n)=I_d\) imply
\begin{align*}
    C_1
    &\leq\sum_{i=1}^n
             \norm{\Xi_i}_{L^2}\norm{\Delta_i}_{L^2}
    \\
    &\leq C(\eta,C_0,\rho_0)
       \left(\sum_{i=1}^n\mathbb E\norm{Y_i}^2\right)^{1/2}
       \left\|
          \left(\sum_{\ell=0}^{n-i}\rho^\ell a_{i+\ell}\right)_{i=1}^n
       \right\|_2
    \\
    &\leq C(\eta,C_0,\rho_0)\sqrt d\,\norm{A}_2.
\end{align*}
Similarly,
\(\norm{\Xi_i}_{L^4}\leq2a_i\), and H\"older's inequality together with
\eqref{eq:Delta-low-moment} gives
\begin{align*}
    C_3
    \leq
    \sum_{i=1}^n
        \norm{\Xi_i}_{L^4}
        \norm{\Delta_i}_{L^4}^3\leq
    C(\eta,C_0,\rho_0)\norm{A}_4
    \left\|
        \left(
            \sum_{\ell=0}^{n-i}
                \rho^\ell a_{i+\ell}
        \right)_{i=1}^n
    \right\|_4^3\leq
    C(\eta,C_0,\rho_0)\norm{A}_4^4.
\end{align*}
This proves \eqref{eq:C1-C3-bounds} and completes the proof of
Lemma~\ref{lem:tensors}.

\subsubsection{Proof of Lemma~\ref{lem:suffix-fluctuation}}
\label{sec:proof-suffix-fluctuation}

We first record two standard estimates in the forms needed below. For two sigma-fields $\mathcal A=\sigma(X)$ and $\mathcal B=\sigma(Y)$, let
$\beta(\mathcal A,\mathcal B)$ denote their absolute-regularity coefficient,
defined by
\[
    \beta(\mathcal A,\mathcal B)
    :=
    \E\!\left[
        \norm{
            \law(Y\mid X)-\law(Y)
        }_{\TV}
    \right].
\]

\begin{lemma}
\label{lem:beta-conditional-expectation}
Let \(1\leq p<s<\infty\), let \(\mathsf H\) be a finite-dimensional real
Hilbert space, and let \(Z\in L^s(\mathsf H)\) be centered and
\(\mathcal B\)-measurable. Then
\begin{equation}
\label{eq:beta-conditional-general}
    \norm{\E[Z\mid\mathcal A]}_{L^p(\mathsf H)}
    \leq
    2\,
    \beta(\mathcal A,\mathcal B)^{1/p-1/s}
    \norm{Z}_{L^s(\mathsf H)}.
\end{equation}
\end{lemma}

\textit{Proof of Lemma~\ref{lem:beta-conditional-expectation}.}
By Berbee's coupling lemma \citep{berbee1979random}, on an extension there
exists a copy \(Z^*\stackrel{\mathrm d}=Z\), independent of
\(\mathcal A\), such that
\[
    \Pp(Z\neq Z^*)
    \leq
    \beta(\mathcal A,\sigma(Z))
    \leq
    \beta(\mathcal A,\mathcal B).
\]
Since \(\E[ Z^*]=\E [Z]=0\), conditional Jensen's inequality and H\"older's
inequality give
\[
    \norm{\E[Z\mid\mathcal A]}_{L^p}
    \leq
    \norm{
        (Z-Z^*)\one_{\{Z\neq Z^*\}}
    }_{L^p}
    \leq
    2\norm{Z}_{L^s}
    \Pp(Z\neq Z^*)^{1/p-1/s},
\]
which proves \eqref{eq:beta-conditional-general}.
\hfill\(\square\)

\begin{lemma}[{\citealp[Eq.~(2.5)]{dedecker2015moment}}]
\label{lem:hilbert-BR}
Let \((d_i,\mathcal G_i)_{i=1}^N\) be a martingale-difference sequence
with values in a finite-dimensional real Hilbert space \(\mathsf H\), and
let \(p\geq2\). If deterministic numbers \(b_i\geq0\) satisfy $\norm{d_i}_{L^p(\mathsf H)}
    \leq
    b_i,$ then
\begin{equation}
\label{eq:hilbert-BR}
    \norm{
        \sum_{i=1}^N d_i
    }_{L^p(\mathsf H)}
    \leq
    \sqrt{p-1}\,\norm{(b_i)_{i=1}^N}_2.
\end{equation}
\end{lemma}

We now prove Lemma~\ref{lem:suffix-fluctuation}.  Set
\[
    s:=(1+\eta/4)p,
    \qquad
    2s=(2+\eta/2)p<q=(2+\eta)p.
\]
For \(0\leq\ell\leq n-i\), Lemma~\ref{lem:block-coupling}(ii) gives
\[
    Y_{i+\ell}^{i,1}-Y_{i+\ell}^{i,0}
    =
    \bigl(
        Y_{i+\ell}^{i,1}-Y_{i+\ell}^{i,0}
    \bigr)
    \one_{\{\kappa_i>\ell\}}.
\]
By Lemma~\ref{lem:block-coupling}(i), both coordinates have the same
marginal law as the stationary chain, and hence
\[
    \norm{
        Y_{i+\ell}^{i,1}-Y_{i+\ell}^{i,0}
    }_{L^q}
    \leq
    2a_{i+\ell}.
\]
Since $\frac1{2s}-\frac1q
    =
    \frac{c_\eta}{p}$ for some constant \(c_\eta>0\), H\"older's inequality and
\eqref{eq:anchored-tail} yield
\begin{equation*}
    \norm{
        Y_{i+\ell}^{i,1}-Y_{i+\ell}^{i,0}
    }_{L^{2s}}
    \leq
    C(\eta,C_0,\rho_0) e^{-C(\eta,C_0,\rho_0)\ell/p}a_{i+\ell}.
\end{equation*}
Define
\[
    \overline a_i
    :=
    \sum_{\ell=0}^{n-i}
        e^{-C(\eta,C_0,\rho_0)\ell/p}a_{i+\ell}.
\]
Minkowski's and H\"older's inequalities, together with
\eqref{eq:suffix-functional-domination}, give
\begin{equation}
\label{eq:suffix-basic-moments}
    \norm{\Xi_i}_{L^{2s}}
    \leq
    2a_i,
    \qquad
    \norm{\Delta_i}_{L^{2s}}
    \leq
    C(\eta,C_0,\rho_0)\overline a_i,
    \qquad
    \norm{\Phi_i}_{L^s(\mathsf H)}
    \leq
    C(\eta,C_0,\rho_0)\Lambda a_i\overline a_i.
\end{equation}
Let
\[
    \Pi_jZ
    :=
    \E[Z\mid\mathcal P_j]
    -
    \E[Z\mid\mathcal P_{j-1}],
    \qquad j\in\mathbb Z,
\]
and set
\[
    U_i:=\E[\Phi_i\mid\mathcal X].
\]
Since \(\Phi_i\) is an integrable functional of the two coupled suffixes,
Lemma~\ref{lem:block-coupling}(iii) implies that \(U_i\) is 
measurable with respect to $\sigma(X_r:r\geq i-1).$ For later use, put
\[
    \mathcal G_{j,m}
    :=
    \sigma(X_r:r\geq j+m),
    \qquad m\geq0.
\]
By the standard characterization of the absolute-regularity coefficient for
stationary Markov chains \citep[Section~3, Equation~(3.1)]{samur2004regularity},
\begin{equation}
\label{eq:future-beta-mixing}
    \beta(\mathcal P_j,\mathcal G_{j,m})=\int_{\mathsf X}
        \norm{P^m(x,\cdot)-\pi}_{\TV}\,
        \pi(\mathrm dx)
    \leq
    C(\eta,C_0,\rho_0) \rho_0^m.
\end{equation}
Combining Lemma~\ref{lem:beta-conditional-expectation} with
\eqref{eq:future-beta-mixing}, and using $\frac1p-\frac1s
    =
    \frac{\eta}{(4+\eta)p},$ shows that every centered \(Z\in L^s(\mathsf H)\), measurable with respect
to \(\mathcal G_{j,m}\), satisfies
\begin{equation}
\label{eq:beta-projection}
    \norm{
        \E[Z\mid\mathcal P_j]
    }_{L^p(\mathsf H)}
    \leq
    C(\eta,C_0,\rho_0) e^{-C(\eta,C_0,\rho_0)m/p}\norm{Z}_{L^s(\mathsf H)}.
\end{equation}
For \(m\geq0\), define
\[
    U_i^{(m)}
    :=
    \E\!\left[
        \Phi_i\one_{\{\kappa_i\leq m\}}
        \,\middle|\,
        \mathcal X
    \right].
\]
The random variable $\Phi_i\one_{\{\kappa_i\leq m\}}$ is
$\mathcal J_{\mathfrak b(m)}^i$-measurable. To see this, for $m\geq1$
define the prefix increment
\[
    \Delta_i^{[m]}:=
    \sum_{\ell=0}^{\min\{m-1,n-i\}}
        \bigl(Y_{i+\ell}^{i,1}-Y_{i+\ell}^{i,0}\bigr).
\]
It is $\mathcal J_m^i$-measurable, and on $\{\kappa_i\leq m\}$,
\eqref{eq:poisson-increment-localization} gives
$\Delta_i=\Delta_i^{[m]}$. Thus
\[
    \Phi_i\one_{\{\kappa_i\leq m\}}
    =\phi_i(\Xi_i,\Delta_i^{[m]})
       \one_{\{\kappa_i\leq m\}}
\]
is $\mathcal J_m^i$-measurable. For $m=0$ it is zero because
$\kappa_i\geq1$. Since $m\leq\mathfrak b(m)$, the asserted
measurability follows.
Lemma~\ref{lem:block-coupling}(iii) therefore yields
\begin{equation}
\label{eq:Ui-m-window}
    U_i^{(m)}
    \quad\text{is measurable with respect to}\quad
    \sigma\!\left(
        X_{i-1},\ldots,
        X_{i-1+\mathfrak b(m)}
    \right).
\end{equation}
Furthermore, conditional Jensen's inequality, H\"older's inequality,
\eqref{eq:suffix-basic-moments}, and
\eqref{eq:anchored-tail} give
\begin{align}
    \norm{U_i-U_i^{(m)}}_{L^p(\mathsf H)}
    \leq
    \norm{
        \Phi_i\one_{\{\kappa_i>m\}}
    }_{L^p(\mathsf H)}&\leq
    \norm{\Phi_i}_{L^s(\mathsf H)}
    \Pp(\kappa_i>m)^{1/p-1/s}
    \nonumber\\
    &\leq
    C(\eta,C_0,\rho_0)\Lambda e^{-C(\eta,C_0,\rho_0)m/p}a_i\overline a_i.
    \label{eq:suffix-localization}
\end{align}
We claim that, for every \(\ell\in\mathbb Z\),
\begin{equation}
\label{eq:suffix-projection-decay}
    \norm{
        \Pi_{i-1+\ell}U_i
    }_{L^p(\mathsf H)}
    \leq
    C(\eta,C_0,\rho_0)\Lambda
    e^{-C(\eta,C_0,\rho_0)|\ell|/p}
    a_i\overline a_i.
\end{equation}
For \(\ell\leq-1\), apply \eqref{eq:beta-projection} to
\(U_i-\E [U_i]\), which is measurable with respect to
\(\sigma(X_r:r\geq i-1)\). Since
\[
    \norm{U_i-\E [U_i]}_{L^s}
    \leq
    2\norm{\Phi_i}_{L^s},
\]
the two conditional expectations defining
\(\Pi_{i-1+\ell}\) give
\[
    \norm{
        \Pi_{i-1+\ell}U_i
    }_{L^p}
    \leq
    C(\eta,C_0,\rho_0)\Lambda
    e^{-C(\eta,C_0,\rho_0)|\ell|/p}
    a_i\overline a_i.
\]
For \(\ell\geq2L\), take $m:=\lfloor\ell/2\rfloor.$ Then
\[
    \mathfrak b(m)
    \leq
    m+L-1
    \leq
    \ell-1.
\]
By \eqref{eq:Ui-m-window}, \(U_i^{(m)}\) is
\(\mathcal P_{i-2+\ell}\)-measurable, and hence
\[
    \Pi_{i-1+\ell}U_i^{(m)}=0.
\]
Consequently,
\[
    \norm{
        \Pi_{i-1+\ell}U_i
    }_{L^p}
    \leq
    2\norm{
        U_i-U_i^{(m)}
    }_{L^p}\leq
    C(\eta,C_0,\rho_0)\Lambda
    e^{-C(\eta,C_0,\rho_0)\ell/p}
    a_i\overline a_i
\]
by \eqref{eq:suffix-localization}. Finally, for the finitely many values
\(0\leq\ell<2L\),
\eqref{eq:suffix-basic-moments} give
\[
    \norm{
        \Pi_{i-1+\ell}U_i
    }_{L^p}
    \leq
    C(\eta,C_0,\rho_0)\Lambda a_i\overline a_i.
\]
Increasing \(C(\eta,C_0,\rho_0)\) if necessary proves
\eqref{eq:suffix-projection-decay} for all \(\ell\in\mathbb Z\).

For
\(r,s_0\geq0\), telescoping gives
\begin{equation}
\label{eq:bilateral-finite-telescope}
    \sum_{\ell=-r}^{s_0}
        \Pi_{i-1+\ell}U_i
    =
    \E[
        U_i\mid\mathcal P_{i-1+s_0}
    ]
    -
    \E[
        U_i\mid\mathcal P_{i-2-r}
    ].
\end{equation}
The first conditional expectation converges to \(U_i\) in \(L^p\) by
upward martingale convergence, since $\sigma\!\left(
        \bigcup_{j\in\mathbb Z}\mathcal P_j
    \right)
    =
    \mathcal X.$ Moreover, \eqref{eq:beta-projection} yields
\[
    \norm{
        \E[
            U_i-\E [U_i]
            \mid
            \mathcal P_{i-2-r}
        ]
    }_{L^p}
    \leq
    C(\eta,C_0,\rho_0) e^{-C(\eta,C_0,\rho_0)(r+1)/p}
    \norm{U_i-\E [U_i]}_{L^s}
    \longrightarrow
    0.
\]
Thus the second conditional expectation converges to
\(\E [U_i]\) in \(L^p\). Since
\eqref{eq:suffix-projection-decay} also implies
\[
    \sum_{\ell\in\mathbb Z}
        \norm{
            \Pi_{i-1+\ell}U_i
        }_{L^p}
    <\infty,
\]
passing to the limits in \eqref{eq:bilateral-finite-telescope} yields
\begin{equation}
\label{eq:bilateral-decomposition}
    U_i-\E [U_i]
    =
    \sum_{\ell\in\mathbb Z}
        \Pi_{i-1+\ell}U_i
    \qquad\text{in }L^p(\mathsf H).
\end{equation}
Fix $\ell\in\mathbb Z$. The variables
$\Pi_{i-1+\ell}U_i$, $1\leq i\leq n$, form a martingale-difference
sequence for the filtration $(\mathcal P_{i-1+\ell})_{i=1}^n$.
Apply Lemma~\ref{lem:hilbert-BR} with
\[
    b_i
    :=
    C(\eta,C_0,\rho_0)\Lambda
    e^{-C(\eta,C_0,\rho_0)|\ell|/p}
    a_i\overline a_i.
\]
Together with \eqref{eq:suffix-projection-decay}, this gives
\begin{equation*}
    \norm{
        \sum_{i=1}^n
            \Pi_{i-1+\ell}U_i
    }_{L^p(\mathsf H)}
    \leq
    C(\eta,C_0,\rho_0)\Lambda
    e^{-C(\eta,C_0,\rho_0)|\ell|/p}
    \sqrt p\,
    \norm{
        (a_i\overline a_i)_{i=1}^n
    }_2.
\end{equation*}
Using \eqref{eq:bilateral-decomposition}, Minkowski's inequality, and
\[
    \sum_{\ell\in\mathbb Z}
        e^{-C(\eta,C_0,\rho_0)|\ell|/p}
    \leq
    C(\eta,C_0,\rho_0)p,
\]
we obtain
\begin{equation}
\label{eq:suffix-U-sum}
    \norm{
        \sum_{i=1}^n
            (U_i-\E[ U_i])
    }_{L^p(\mathsf H)}
    \leq
    C(\eta,C_0,\rho_0)\Lambda
    p^{3/2}
    \norm{
        (a_i\overline a_i)_{i=1}^n
    }_2.
\end{equation}
Young's convolution inequality gives, for every \(r\geq1\),
\[
    \norm{
        (\overline a_i)_{i=1}^n
    }_r
    \leq
    \left(
        \sum_{\ell\geq0}
            e^{-C(\eta,C_0,\rho_0)\ell/p}
    \right)
    \norm{A}_r
    \leq
    C(\eta,C_0,\rho_0)p\norm{A}_r.
\]
Consequently,
\[
    \norm{
        (a_i\overline a_i)_{i=1}^n
    }_2
    \leq
    C(\eta,C_0,\rho_0)p\norm{A}_4^2.
\]
Substituting this bound into \eqref{eq:suffix-U-sum} yields
\[
    \norm{
        \sum_{i=1}^n
            (U_i-\E [U_i])
    }_{L^p(\mathsf H)}
    \leq
    C(\eta,C_0,\rho_0)\Lambda
    p^{5/2}\norm{A}_4^2
    =
    C(\eta,C_0,\rho_0)\Lambda B_p.
\]
Finally, \(W\) is \(\mathcal X\)-measurable, so the tower property gives
\[
\begin{aligned}
    \sum_{i=1}^n
        \bigl(
            \E[\Phi_i\mid W]-\E[\Phi_i]
        \bigr)
    &=
    \E\!\left[
        \sum_{i=1}^n
            (U_i-\E [U_i])
        \,\middle|\,
        W
    \right].
\end{aligned}
\]
Conditional Jensen's inequality now proves
\eqref{eq:suffix-fluctuation}.

\section{Conclusion}

We establish higher-order Wasserstein Gaussian approximation
bounds for deterministic matrix-weighted sums of multivariate
Poisson martingale differences from uniformly
ergodic Markov chains. The bounds explicitly track the
Wasserstein order $p$ and dimension $d$, and attain the
optimal $O(n^{-1/2})$ rate in the balanced-increment regime
for fixed $p$ and $d$. Equal scalar weights and the Poisson
decomposition yield the same optimal rate for additive
functionals of Markov chains. 

The analysis combines a tensor-level formulation of the
Ornstein--Uhlenbeck relative-score approach of
\citet{fang2023p} with a refresh-then-maximal coupling
tailored to Markovian dependence.
This construction separates the Stein identity from
control of the coupling increment while preserving the
conditional structure needed for sharp tensor estimates.

The matrix-weighted formulation directly matches the
martingale sums arising in the error decompositions of
nonlinear Markovian SA. Applying our bounds to these sums,
and comparing Gaussian covariances only afterward, yields
sharp $\mathcal W_p$ bounds for both last iterates and
Polyak--Ruppert averages under the stated assumptions.

\paragraph{Acknowledgments.} Y.\ Zhang and Q.\ Xie are supported in part by NSF grants CNS-1955997, EPCN-2339794, and EPCN-2432546.

\bibliographystyle{plainnat}
\bibliography{main}

\appendix

\section{Proof of Lemma~\ref{lem:Lr-Poisson}}
\label{sec:proof-poisson}
Stationarity and Jensen's inequality show that both \(P^m\) and \(\Pi\)
are contractions on \(L^1(\pi;\R^d)\). Hence
\begin{equation}
\label{eq:L1-Pm-Pi}
    \norm{(P^m-\Pi)f}_{L^1(\pi)}
    \leq
    2\norm{f}_{L^1(\pi)}.
\end{equation}
For \(f\in L^\infty(\pi;\R^d)\), the dual representation of the Euclidean norm and
\eqref{eq:UGE} then give, for $\pi$-almost every $x\in\mathsf X$,
\begin{align*}
    \norm{(P^m-\Pi)f(x)}
    &=
    \sup_{\norm{u}\leq1}
    \left|
        \int_{\mathsf X}
            \ip{u}{f(y)}
            \bigl(P^m(x,\mathrm dy)-\pi(\mathrm dy)\bigr)
    \right|
    \\
    &\leq
    2\norm{f}_{L^\infty(\pi)}
    \norm{P^m(x,\cdot)-\pi}_{\TV}\leq
    2C_0\rho_0^m\norm{f}_{L^\infty(\pi)}.
\end{align*}
Thus
\begin{equation}
\label{eq:Linf-Pm-Pi}
    \norm{(P^m-\Pi)f}_{L^\infty(\pi)}
    \leq
    2C_0\rho_0^m\norm{f}_{L^\infty(\pi)}.
\end{equation}
Riesz–Thorin theorem,
applied to \(P^m-\Pi\) between
\eqref{eq:L1-Pm-Pi} and \eqref{eq:Linf-Pm-Pi}, yields
\eqref{eq:Lr-uniform-decay}. Now suppose that \(\pi(f)=0\). Then \(P^m f=(P^m-\Pi)f\), so
\eqref{eq:Lr-uniform-decay} gives
\[
    \sum_{m=0}^{\infty}\norm{P^m f}_{L^r(\pi)}
    \leq
    2C_0^{1-1/r}
    \sum_{m=0}^{\infty}\rho_0^{m(1-1/r)}
    \norm{f}_{L^r(\pi)}
    <\infty.
\]
Hence the series converges in \(L^r(\pi;\R^d)\). Let
\(g_{f,N}:=\sum_{m=0}^N P^m f\), and let \(g_f\) denote its \(L^r\)-limit.
Jensen's inequality and invariance of \(\pi\) also show that \(P\) is a
contraction on \(L^r(\pi;\R^d)\). Hence
\(Pg_{f,N}\to Pg_f\) in \(L^r\), while
\[
    g_{f,N}-Pg_{f,N}
    =
    f-P^{N+1}f
    \longrightarrow
    f
    \qquad\text{in }L^r(\pi;\R^d).
\]
Therefore \(g_f-Pg_f=f\) in \(L^r\). Moreover, invariance of \(\pi\) gives
\(\pi(P^m f)=\pi(f)=0\), so \(\pi(g_f)=0\).

\section{Proof of Corollary \ref{cor:additive-functional}}\label{sec:proof-additive-functional}
The Poisson equation \(g-Pg=h\) gives, for every \(i\geq1\),
\begin{align*}
    h(X_i)
    =
    g(X_i)-Pg(X_i)=
    D_i
    +
    Pg(X_{i-1})-Pg(X_i).
\end{align*}
Summing over \(i\) therefore yields the martingale 
decomposition
\begin{equation}
\label{eq:additive-martingale-decomposition}
    \mathsf S_n
    =
    \sum_{i=1}^n D_i
    +
    Pg(X_0)-Pg(X_n).
\end{equation}
Write
\[
    \mathcal M_n
    :=
    \sum_{i=1}^n D_i,
    \qquad
    B_n
    :=
    Pg(X_0)-Pg(X_n),
\]
so that \(\mathsf S_n=\mathcal M_n+B_n\). We first identify the covariance of the martingale component. Since
\((D_i)_{i\geq1}\) is a stationary martingale-difference sequence,
\[
    \operatorname{Cov}(\mathcal M_n)
    =
    n\Sigma_D,
    \qquad
    \Sigma_D
    :=
    \E[D_1D_1^\top].
\]
Moreover, stationarity  gives
\[
    \sup_{n\geq1}
    \norm{B_n}_{L^2}
    \leq
    2\norm{Pg}_{L^2(\pi)}
    \leq
    2\norm{g}_{L^2(\pi)}
    <\infty.
\]
It follows from \eqref{eq:additive-martingale-decomposition} and the
Cauchy--Schwarz inequality that
\[
    \frac{1}{n}
    \operatorname{Cov}(\mathsf S_n)
    -
    \Sigma_D
    \longrightarrow
    0.
\]
Indeed, the covariance of \(B_n\), divided by \(n\), is \(O(n^{-1})\),
while the cross-covariance between \(\mathcal M_n\) and \(B_n\), divided
by \(n\), is \(O(n^{-1/2})\). Hence
\eqref{eq:additive-asymptotic-covariance} implies
\begin{equation}
\label{eq:D-covariance-identity}
    \E[D_1D_1^\top]
    =
    I_d.
\end{equation}
Now define $\widetilde S_n
    :=
    \frac{1}{\sqrt n}
    \sum_{i=1}^n D_i.$ By \eqref{eq:D-covariance-identity}, $\operatorname{Cov}(\widetilde S_n)
    =
    I_d.$ We may therefore apply Theorem~\ref{thm:main} with $M_{n,i}
    =
    \frac{1}{\sqrt n}I_d$ and $Y_i
    =
    \frac{1}{\sqrt n}D_i.$ By stationarity,
\[
    a_i
    =
    \norm{Y_i}_{L^q}
    =
    \frac{b_q}{\sqrt n},
    \qquad 1\leq i\leq n,
\]
where $b_q
    :=
    \norm{
        g(X_1)-Pg(X_0)
    }_{L^q}$. Consequently,
\begin{equation}
\label{eq:additive-A-norms}
    \norm{A}_2=b_q,
    \qquad
    \norm{A}_4^2
    =
    \frac{b_q^2}{\sqrt n}.
\end{equation}
Substituting \eqref{eq:additive-A-norms} into
\eqref{eq:main-bound} gives
\begin{align}
\label{eq:additive-martingale-bound}
\W_p\!\left(
    \law(\widetilde S_n),
    \mathcal N(0,I_d)
\right)
\leq
C(\eta,C_0,\rho_0)
\left[
    \frac{
        p^3b_q^2
        +
        p\,d^{1/4}b_q^{5/2}
    }{\sqrt n}
\right].
\end{align}
It remains to control the boundary term in
\eqref{eq:additive-martingale-decomposition}. Using the coupling in which
\(\mathcal M_n\) and \(B_n\) are defined on the original probability
space, the triangle inequality for \(\W_p\) yields
\begin{align*}
\W_p\!\left(
    \law\!\left(
        \mathsf S_n/\sqrt{n}
    \right),
    \mathcal N(0,I_d)
\right)\leq
\W_p\!\left(
    \law(\widetilde S_n),
    \mathcal N(0,I_d)
\right)
+
\frac{1}{\sqrt n}
\norm{B_n}_{L^p}.
\end{align*}
Since \(q>p\), stationarity and Jensen's inequality imply
\[
    \norm{B_n}_{L^p}
    \leq
    2\norm{Pg}_{L^p(\pi)}
    \leq
    2\norm{Pg}_{L^q(\pi)}
    \leq
    2\norm{g}_{L^q(\pi)}.
\]
Finally,
\begin{align*}
    \norm{
        g(X_1)-Pg(X_0)
    }_{L^q}\leq
    2\norm{g}_{L^q(\pi)}\leq
    C(\eta,C_0,\rho_0)
    \norm{h}_{L^q(\pi)},
\end{align*}
where the last inequality follows from Lemma \ref{lem:Lr-Poisson}. Combining the preceding bounds completes the proof of Corollary~\ref{cor:additive-functional}.

\section{Proofs for the Stochastic Approximation Application}
\label{sec:proof-sa-application}

Write
\[
    e_k:=\theta_k-\theta^*,
    \qquad
    \mathcal G_k:=\sigma(\theta_1,X_0,\ldots,X_k).
\]
The recursion implies that $e_k$ is
$\mathcal G_{k-1}$-measurable. Recall that $\gamma_k=\gamma_1k^{-a}$ and
\[
    \Phi_{t,s}:=(I_d-\gamma_tH)\cdots(I_d-\gamma_sH),
    \qquad t\geq s,
\]
with $\Phi_{t,s}:=I_d$ when $t<s$. The response weights are
\begin{equation}
\label{eq:sa-response-weights}
    L_{n,k}:=\frac{\gamma_k}{\sqrt{\gamma_n}}\Phi_{n,k+1},
    \qquad
    Q_{n,k}:=\gamma_k\sum_{j=k}^n\Phi_{j,k+1},
    \qquad
    R_{n,k}:=\frac{Q_{n,k}}{\sqrt n}.
\end{equation}
Their associated covariances are
\[
    V_n:=\sum_{k=1}^nL_{n,k}\Gamma L_{n,k}^{\top},
    \qquad
    \overline V_n:=\sum_{k=1}^nR_{n,k}\Gamma R_{n,k}^{\top}.
\]
We also recall the limiting covariances
\[
    H\Sigma+\Sigma H^\top=\Gamma,
    \qquad
    \overline\Sigma=H^{-1}\Gamma H^{-\top}.
\]
Both are positive definite because $H$ is positive stable
and $\Gamma\succ0$.
For deterministic matrices $w_1,\ldots,w_n$, define
\[
    v(w):=\norm{w_1}_{\mathrm{op}}
          +\norm{w_n}_{\mathrm{op}}
          +\sum_{k=1}^{n-1}
             \norm{w_{k+1}-w_k}_{\mathrm{op}}.
\]

\subsection{Deterministic estimates for the response weights}

\begin{lemma}[Polynomial stepsizes]
\label{lem:sa-polynomial-weights}
Let $0<a<1$, and suppose that the eigenvalues of $H$ have
positive real parts.
The weights in~\eqref{eq:sa-response-weights} satisfy
\begin{equation}\label{eq:sa-last-weight-estimates}
\begin{aligned}
    \max_{k\leq n}\norm{L_{n,k}}_{\mathrm{op}}
        &\lesssim\sqrt{\gamma_n},
    &\sum_{k=1}^n\norm{L_{n,k}}_{\mathrm{op}}^2
        &\lesssim1,
    &\sum_{k=1}^n\norm{L_{n,k}}_{\mathrm{op}}^4
        &\lesssim\gamma_n,
    \\
    v(L_{n,\cdot})
    +\sum_{k=1}^n\norm{L_{n,k}}_{\mathrm{op}}\gamma_k
        &\lesssim\sqrt{\gamma_n},
    &\sum_{k=1}^n\norm{L_{n,k}}_{\mathrm{op}}^2\gamma_k
        &\lesssim\gamma_n.
\end{aligned}
\end{equation}
For the averaging weights,
\begin{equation}\label{eq:sa-pr-weight-estimates}
\begin{aligned}
    \max_{k\leq n}\norm{R_{n,k}}_{\mathrm{op}}
        &\lesssim n^{-1/2},
    &\sum_{k=1}^n\norm{R_{n,k}}_{\mathrm{op}}^2
        &\lesssim1,
    &\sum_{k=1}^n\norm{R_{n,k}}_{\mathrm{op}}^4
        &\lesssim n^{-1},
    \\
    v(R_{n,\cdot})
    +\sum_{k=1}^n\norm{R_{n,k}}_{\mathrm{op}}\gamma_k
        &\lesssim\sqrt n\,\gamma_n,
    &\sum_{k=1}^n\norm{R_{n,k}}_{\mathrm{op}}^2\gamma_k
        &\lesssim\gamma_n.
\end{aligned}
\end{equation}
Moreover,
\begin{equation}
\label{eq:sa-proof-covariance-estimates}
    \norm{V_n-\Sigma}_{\HS}
        \lesssim\gamma_n+\frac1{n\gamma_n},
    \qquad
    \norm{\overline V_n-\overline\Sigma}_{\HS}
        \lesssim\frac1{n\gamma_n}.
\end{equation}
The initialization weights satisfy
\begin{equation}
\label{eq:sa-initialization-polynomial}
    \gamma_n^{-1/2}\norm{\Phi_{n,1}}_{\mathrm{op}}
        \lesssim\sqrt{\gamma_n},
    \qquad
    \sum_{j=1}^{\infty}\norm{\Phi_{j,1}}_{\mathrm{op}}
        <\infty.
\end{equation}
\end{lemma}

\begin{proof}
Choose a positive definite matrix $B$ solving
$H^\top B+BH=I_d$.
In the corresponding weighted Euclidean norm,
$I_d-\gamma_kH$ is a contraction by a factor at most
$1-c\gamma_k$ for all sufficiently large $k$.
The finitely many earlier factors can be absorbed into a
constant. Equivalence of norms therefore gives
\begin{equation}
\label{eq:sa-product-bound}
    \norm{\Phi_{t,s}}_{\mathrm{op}}
    \leq
    C\exp\!\left(-c\sum_{j=s}^t\gamma_j\right),
    \qquad t\geq s.
\end{equation}
For any fixed $b>0$ and $c'>0$,
\begin{equation}
\label{eq:sa-backward-convolution}
    \sum_{k=1}^n\gamma_k^b
       \exp\!\left(-c'\sum_{j=k+1}^n\gamma_j\right)
    \lesssim\gamma_n^{b-1}.
\end{equation}
Indeed, for $k>n/2$, the stepsizes are comparable to
$\gamma_n$, and the exponential is bounded by
$e^{-c'(n-k)\gamma_n}$.
Summing this geometric sequence gives $C\gamma_n^{b-1}$.
For $k\leq n/2$, the exponential is at most
$e^{-c''n^{1-a}}$, and the sum of the remaining factors
grows at most polynomially. This contribution is smaller
than the claimed bound after increasing its constant.
The same splitting also bounds the largest individual
weighted term.

We will also use the forward version: for each fixed
$b\in\R$,
\begin{equation}
\label{eq:sa-forward-convolution}
    \sum_{j=k}^{\infty}j^b
       \exp\!\left(-c'\sum_{\ell=k+1}^j\gamma_\ell\right)
    \lesssim k^{a+b}.
\end{equation}
For $k\leq j\leq2k$, we have $j^b\lesssim k^b$ and
\[
    \sum_{\ell=k+1}^j\gamma_\ell
    \gtrsim(j-k)k^{-a}.
\]
Summing the resulting geometric sequence gives $Ck^{a+b}$.
For $j>2k$, partition the indices into intervals
$2^m k<j\leq2^{m+1}k$, $m\geq1$.
The contribution from such an interval is bounded by
\[
    C(2^m k)^{b+1}e^{-c''(2^m k)^{1-a}}.
\]
After division by $k^{a+b}$, these bounds are summable
in $m$, uniformly in $k\geq1$.
This proves~\eqref{eq:sa-forward-convolution}.

\medskip
\noindent
\textit{Last-iterate weights.}
For $r=2,4$, equations~\eqref{eq:sa-product-bound}
and~\eqref{eq:sa-backward-convolution} give
\[
    \sum_{k=1}^n\norm{L_{n,k}}_{\mathrm{op}}^r
    \lesssim
    \gamma_n^{-r/2}\gamma_n^{r-1}
    =\gamma_n^{r/2-1}.
\]
The maximum bound follows by the same early/late split.
Likewise,
\[
    \sum_{k=1}^n\norm{L_{n,k}}_{\mathrm{op}}\gamma_k
        \lesssim\gamma_n^{-1/2}\gamma_n,
    \qquad
    \sum_{k=1}^n\norm{L_{n,k}}_{\mathrm{op}}^2\gamma_k
        \lesssim\gamma_n^{-1}\gamma_n^2.
\]
For $k<n$, direct subtraction yields
\begin{equation}
\label{eq:sa-last-weight-difference}
    L_{n,k+1}-L_{n,k}
    =
    \gamma_n^{-1/2}\Phi_{n,k+2}
    \bigl[
        (\gamma_{k+1}-\gamma_k)I_d
        +\gamma_k\gamma_{k+1}H
    \bigr].
\end{equation}
Since
$|\gamma_{k+1}-\gamma_k|\lesssim\gamma_k/k$
and $k^{-1}\lesssim\gamma_k$, another application of
\eqref{eq:sa-backward-convolution}, with a shift of one
index, bounds the sum of these differences by
$C\sqrt{\gamma_n}$.
The endpoint weights satisfy the same bound.
This proves~\eqref{eq:sa-last-weight-estimates}.
Equation~\eqref{eq:sa-initialization-polynomial}
follows from
$\sum_{k=1}^n\gamma_k\gtrsim n^{1-a}$
and~\eqref{eq:sa-product-bound}.

\medskip
\noindent
\textit{Averaging weights and their covariance.}
Define the absolutely convergent infinite-horizon weight
\[
    Q_{\infty,k}
    :=
    \gamma_k\sum_{j=k}^{\infty}\Phi_{j,k+1}.
\]
Equation~\eqref{eq:sa-forward-convolution} with $b=0$
gives
\[
    \sup_k\norm{Q_{\infty,k}}_{\mathrm{op}}<\infty,
    \qquad
    \sup_{n\geq1}\max_{k\leq n}
        \norm{Q_{n,k}}_{\mathrm{op}}<\infty.
\]
All factors in $\Phi$ are polynomials in $H$ and commute.
Using
\[
    \Phi_{j,k+1}-\Phi_{j+1,k+1}
    =
    \gamma_{j+1}H\Phi_{j,k+1}
\]
and summation by parts, we obtain
\begin{equation*}
    HQ_{\infty,k}
    =
    \frac{\gamma_k}{\gamma_{k+1}}I_d
    +\gamma_k\sum_{j=k+1}^{\infty}
        \left(
            \frac1{\gamma_{j+1}}-\frac1{\gamma_j}
        \right)\Phi_{j,k+1}.
\end{equation*}
The boundary term at infinity vanishes
by~\eqref{eq:sa-product-bound}.
Since
\[
    \frac{\gamma_k}{\gamma_{k+1}}-1=O(k^{-1}),
    \qquad
    \gamma_{j+1}^{-1}-\gamma_j^{-1}=O(j^{a-1}),
\]
equation~\eqref{eq:sa-forward-convolution} yields
\begin{equation*}
    \norm{Q_{\infty,k}-H^{-1}}_{\mathrm{op}}
    \lesssim k^{a-1}.
\end{equation*}
The truncated tail has the exact expression
\[
    Q_{\infty,k}-Q_{n,k}
    =
    \frac{\gamma_k}{\gamma_{n+1}}
    \Phi_{n+1,k+1}Q_{\infty,n+1}.
\]
For any $\varepsilon>0$, there is a constant
$C_\varepsilon$ such that
\[
    \frac{\gamma_k}{\gamma_{n+1}}
    \leq
    C_\varepsilon
    \exp\!\left(
        \varepsilon\sum_{j=k+1}^{n+1}\gamma_j
    \right).
\]
To see this, write the logarithm of the ratio as the sum
of $a\log(j/(j-1))$, and use
\[
    a\log(j/(j-1))\leq\varepsilon\gamma_j
\]
for all sufficiently large $j$; the remaining finite sum
is absorbed into the constant.
Taking $\varepsilon$ smaller than the contraction exponent
in~\eqref{eq:sa-product-bound} gives
\begin{equation}
\label{eq:sa-finite-weight-error}
    \norm{Q_{n,k}-H^{-1}}_{\mathrm{op}}
    \lesssim
    k^{a-1}
    +\exp\!\left(-c\sum_{j=k+1}^n\gamma_j\right).
\end{equation}
Consequently,
\begin{equation}
\label{eq:sa-summed-weight-error}
    \sum_{k=1}^n\norm{Q_{n,k}-H^{-1}}_{\mathrm{op}}
    \lesssim n^a,
\end{equation}
because $\sum_{k=1}^n k^{a-1}\lesssim n^a$, while
\[
    \sum_{k=1}^n
    \exp\!\left(-c\sum_{j=k+1}^n\gamma_j\right)
    \leq
    \sum_{\ell\geq0}e^{-c\ell\gamma_n}
    \lesssim\gamma_n^{-1}.
\]
Since $Q_{n,k}$ and $H^{-1}$ are uniformly bounded,
\[
    \norm{\overline V_n-\overline\Sigma}_{\HS}
    \leq
    \frac Cn\sum_{k=1}^n
        \norm{Q_{n,k}-H^{-1}}_{\mathrm{op}}
    \lesssim n^{a-1}
    \lesssim\frac1{n\gamma_n}.
\]
The first three estimates
in~\eqref{eq:sa-pr-weight-estimates} follow from
$\norm{R_{n,k}}_{\mathrm{op}}\lesssim n^{-1/2}$.
For the variation, the defining sums imply
\[
    Q_{n,k}
    =
    \gamma_k I_d
    +\frac{\gamma_k}{\gamma_{k+1}}
        (I_d-\gamma_{k+1}H)Q_{n,k+1},
    \qquad k<n.
\]
It follows that
\[
    \norm{Q_{n,k+1}-Q_{n,k}}_{\mathrm{op}}
    \lesssim\gamma_k+k^{-1}
    \lesssim\gamma_k,
\]
and hence
\[
    v(R_{n,\cdot})
    \lesssim
    n^{-1/2}\left(1+\sum_{k=1}^n\gamma_k\right)
    \lesssim\sqrt n\,\gamma_n.
\]
The remaining two weighted sums
in~\eqref{eq:sa-pr-weight-estimates} follow from
$\sum_{k=1}^n\gamma_k\lesssim n\gamma_n$.

\medskip
\noindent\textit{Last-iterate covariance.}
Define
\[
    C_n
    :=
    \sum_{k=1}^n\gamma_k^2
        \Phi_{n,k+1}\Gamma\Phi_{n,k+1}^\top,
    \qquad
    V_n=C_n/\gamma_n.
\]
With
$A_n:=I_d-\gamma_nH$
and $r_n:=\gamma_{n-1}/\gamma_n$, we have, for $n\geq2$,
\[
    V_n
    =
    r_nA_nV_{n-1}A_n^\top+\gamma_n\Gamma.
\]
Writing $E_n:=V_n-\Sigma$ and using
$H\Sigma+\Sigma H^\top=\Gamma$ gives
\begin{equation}
\label{eq:sa-last-covariance-recursion}
    E_n
    =
    r_nA_nE_{n-1}A_n^\top+T_n,
    \qquad
    \norm{T_n}_{\HS}\lesssim n^{-1}+\gamma_n^2.
\end{equation}
Indeed,
\[
    T_n
    =
    (r_n-1)\Sigma
    +\gamma_n(1-r_n)\Gamma
    +r_n\gamma_n^2H\Sigma H^\top,
\]
and $r_n-1=O(n^{-1})$. Use the matrix norm
$\norm{B^{1/2}EB^{1/2}}_{\HS}$.
Because $n^{-1}=o(\gamma_n)$, the linear map
$E\mapsto r_nA_nEA_n^\top$ contracts this norm by at most
$1-c\gamma_n$ for all sufficiently large $n$.
Iterating~\eqref{eq:sa-last-covariance-recursion}
and applying~\eqref{eq:sa-backward-convolution}
with $b=2$ and $b=1/a$ bounds the two forcing terms by
$C\gamma_n$ and $C(n\gamma_n)^{-1}$, respectively.
The initial transient is exponentially small in $n^{1-a}$.
Norm equivalence proves the first bound
in~\eqref{eq:sa-proof-covariance-estimates}.
\end{proof}

\begin{lemma}[Harmonic stepsizes]
\label{lem:sa-harmonic-weights}
Suppose $a=1$ and
\[
    \gamma_1
    \min_{\lambda\in\operatorname{spec}(H)}
        \operatorname{Re}\lambda>1.
\]
Let $H_1:=H-(2\gamma_1)^{-1}I_d$, and let $\Sigma_1\succ0$
solve
\[
    H_1\Sigma_1+\Sigma_1H_1^\top=\Gamma.
\]
Then~\eqref{eq:sa-last-weight-estimates} remains valid, and
\begin{equation}
\label{eq:sa-harmonic-deterministic}
    \norm{V_n-\Sigma_1}_{\HS}\lesssim n^{-1},
    \qquad
    \gamma_n^{-1/2}\norm{\Phi_{n,1}}_{\mathrm{op}}
        \lesssim n^{-1/2}.
\end{equation}
\end{lemma}

\begin{proof}
Choose $b$ such that $1<b<
    \gamma_1\min_{\lambda\in\operatorname{spec}(H)}
        \operatorname{Re}\lambda.$
A Lyapunov norm for a slightly larger spectral shift gives
\[
    \norm{I_d-\gamma_1H/k}\leq1-b/k
\]
for all sufficiently large $k$.
Absorbing the finitely many initial factors therefore gives
\[
    \norm{\Phi_{n,k+1}}_{\mathrm{op}}
        \lesssim(k/n)^b,
    \qquad
    \norm{\Phi_{n,1}}_{\mathrm{op}}
        \lesssim n^{-b}.
\]
In particular,
\[
    \norm{L_{n,k}}_{\mathrm{op}}
    \lesssim n^{1/2-b}k^{b-1}.
\]
Consequently,
\[
    \max_{k\leq n}\norm{L_{n,k}}_{\mathrm{op}}
    \lesssim n^{-1/2},
\]
and, for $r=2,4$,
\[
    \sum_{k=1}^n\norm{L_{n,k}}_{\mathrm{op}}^r
    \lesssim
    n^{r/2-rb}\sum_{k=1}^n k^{r(b-1)}
    \lesssim n^{1-r/2}.
\]
For the stepsize-weighted sums,
\[
    \sum_{k=1}^n\norm{L_{n,k}}_{\mathrm{op}}\gamma_k
    \lesssim
    n^{1/2-b}\sum_{k=1}^n k^{b-2}
    \lesssim n^{-1/2},
\]
and
\[
    \sum_{k=1}^n\norm{L_{n,k}}_{\mathrm{op}}^2\gamma_k
    \lesssim
    n^{1-2b}\sum_{k=1}^n k^{2b-3}
    \lesssim n^{-1}.
\]
Equation~\eqref{eq:sa-last-weight-difference} gives
\[
    \norm{L_{n,k+1}-L_{n,k}}_{\mathrm{op}}
    \lesssim n^{1/2-b}k^{b-2},
\]
so the sum of these differences is $O(n^{-1/2})$.
The initialization estimate follows from $b>1$.

For the covariance, use the preceding covariance recursion
with $r_n=n/(n-1)$ and $\gamma_n=\gamma_1/n$.
Since
\[
    H\Sigma_1+\Sigma_1H^\top
    =
    \Gamma+\gamma_1^{-1}\Sigma_1,
\]
the error $E_n:=V_n-\Sigma_1$ satisfies
\[
    E_n
    =
    r_nA_nE_{n-1}A_n^\top
    +\frac{\gamma_1}{n(n-1)}
        \bigl(\gamma_1H\Sigma_1H^\top-\Gamma\bigr).
\]
In the matrix norm induced by the Lyapunov norm above,
the first term contracts by at most $1-c/n$ for all
sufficiently large $n$, where $c$ can be chosen in
$(1,2b-1)$.
An induction in the scalar inequality
\[
    u_n\leq(1-c/n)u_{n-1}+C/n^2
\]
gives $u_n\lesssim n^{-1}$.
Indeed, substituting $K/(n-1)$ for $u_{n-1}$ gives an
upper bound of $K/n$ when $K$ is sufficiently large.
Increasing $K$ handles the finitely many earlier indices.
This proves the covariance bound.
\end{proof}

\subsection{Control of the nonlinear and Markovian remainders}

\begin{lemma}[A weighted remainder bound]
\label{lem:sa-weighted-remainder}
Under Assumptions~\ref{assumption:MC} and
\ref{ass:sa-application}, define
\begin{align*}
    r(\theta)
        &:=\overline F(\theta)-H(\theta-\theta^*),\\
    b_\theta(x)
        &:=F(\theta,x)-h(x)-\overline F(\theta),\\
    T_k
        &:=Pg(X_k),\\
    \zeta_k
        &:=T_{k-1}-T_k
           +b_{\theta_k}(X_k)+r(\theta_k).
\end{align*}
For every deterministic sequence
$w_1,\ldots,w_n\in\R^{d\times d}$,
\begin{equation}
\label{eq:sa-weighted-remainder-bound}
    \norm{\sum_{k=1}^n w_k\zeta_k}_{L^p}
    \lesssim
    v(w)
    +\left(
        \sum_{k=1}^n
        \norm{w_k}_{\mathrm{op}}^2\gamma_k
    \right)^{1/2}
    +\sum_{k=1}^n
        \norm{w_k}_{\mathrm{op}}\gamma_k.
\end{equation}
\end{lemma}

\begin{proof}
Assumption \ref{ass:sa-application} implies that $\overline F$ is globally
$L_0$-Lipschitz.
Since $\overline F(\theta^*)=0$,
\[
    \norm{r(\theta)}
    \leq
    \bigl(L_0+\norm{H}_{\mathrm{op}}\bigr)
    \norm{\theta-\theta^*},\qquad \forall \theta \in \R^d.
\]
For $\norm{\theta-\theta^*}\leq\delta$, use
Assumption~\ref{ass:sa-application}(iii).
Outside this ball, use $\norm{\theta-\theta^*}
    \leq
    \delta^{-1}\norm{\theta-\theta^*}^2.$ Combining the two regions gives
\begin{equation*}
    \norm{r(\theta)}
    \leq
    \max\!\left\{
        L_1,
        \frac{L_0+\norm{H}_{\mathrm{op}}}{\delta}
    \right\}
    \norm{\theta-\theta^*}^2,
    \qquad \forall \theta \in \R^d.
\end{equation*}
Thus Assumption~\ref{ass:sa-application}(v) implies
\begin{equation}
\label{eq:sa-elementary-moments}
    \norm{e_k}_{L^p}
        \lesssim\sqrt{\gamma_k},
    \qquad
    \norm{r(\theta_k)}_{L^p}
        \lesssim\gamma_k,
    \qquad
    \norm{e_{k+1}-e_k}_{L^p}
        \lesssim\gamma_k.
\end{equation}
For the last inequality, stationarity and $\norm{F(\theta_k,X_k)}
    \leq
    \norm{h(X_k)}+L_0\norm{e_k}$ give
\[
    \norm{e_{k+1}-e_k}_{L^p}
    \leq
    \gamma_k
    \bigl(
        \norm{h}_{L^p(\pi)}
        +L_0\norm{e_k}_{L^p}
    \bigr)\lesssim\gamma_k.
\]
For each $\theta$, $b_\theta$ is centered under $\pi$,
and $b_{\theta^*}=0$.
Moreover,
\begin{equation*}
    \norm{b_\theta-b_{\theta'}}_\infty
    \leq2L_0\norm{\theta-\theta'},
    \qquad
    \norm{b_\theta}_\infty
    \leq2L_0\norm{\theta-\theta^*},\qquad \forall \theta, \theta' \in \R^d.
\end{equation*}
Here $\norm{\cdot}_\infty$ denotes the pointwise supremum
of the Euclidean norm over $x\in\mathsf X$.
Define
\begin{equation}
\label{eq:sa-parameter-poisson}
    G_\theta
    :=\sum_{m=0}^\infty P^m b_\theta.
\end{equation}
By Assumption~\ref{assumption:MC}, the series
in~\eqref{eq:sa-parameter-poisson}
converges uniformly in $x$, locally uniformly in $\theta$. Then, we have
\begin{equation}
\label{eq:sa-parameter-poisson-bounds}
\begin{aligned}
    G_\theta-PG_\theta
        &=b_\theta,\\
    \norm{G_\theta}_\infty+\norm{PG_\theta}_\infty
        &\lesssim\norm{\theta-\theta^*},\\
    \norm{G_\theta-G_{\theta'}}_\infty
    +\norm{PG_\theta-PG_{\theta'}}_\infty
        &\lesssim\norm{\theta-\theta'}.
\end{aligned}
\end{equation}
These bounds hold simultaneously for all
$\theta,\theta' \in \R^d$. Set $E_k
    :=
    G_{\theta_k}(X_k)-PG_{\theta_k}(X_{k-1}).$ Since $\theta_k$ is $\mathcal G_{k-1}$-measurable,
\eqref{eq:sa-parameter-poisson-bounds} implies
\[
    \E[E_k\mid\mathcal G_{k-1}]=0,
    \qquad
    \norm{E_k}_{L^p}
    \lesssim\norm{e_k}_{L^p}
    \lesssim\sqrt{\gamma_k}.
\]
Thus $(E_k,\mathcal G_k)$ is a martingale-difference
sequence, and
\begin{equation*}
    b_{\theta_k}(X_k)
    =
    E_k
    +PG_{\theta_k}(X_{k-1})
    -PG_{\theta_k}(X_k).
\end{equation*}
Then, Lemma~\ref{lem:hilbert-BR} gives
\begin{equation}
\label{eq:sa-adaptive-martingale-bound}
    \norm{\sum_{k=1}^n w_kE_k}_{L^p}
    \lesssim
    \left(
        \sum_{k=1}^n
        \norm{w_k}_{\mathrm{op}}^2\gamma_k
    \right)^{1/2}.
\end{equation}
Summation by parts gives the exact identity
\begin{align*}
    \sum_{k=1}^n w_k
    \bigl[
        PG_{\theta_k}(X_{k-1})
        -PG_{\theta_k}(X_k)
    \bigr]&=
    w_1PG_{\theta_1}(X_0)
    -w_nPG_{\theta_n}(X_n)+
    \sum_{k=1}^{n-1}
        (w_{k+1}-w_k)PG_{\theta_k}(X_k)\\
        &+
    \sum_{k=1}^{n-1}w_{k+1}
    \bigl[
       PG_{\theta_{k+1}}(X_k)
        -PG_{\theta_k}(X_k)
    \bigr].
\end{align*}
By~\eqref{eq:sa-parameter-poisson-bounds} and
$\sup_k\norm{e_k}_{L^p}<\infty$, the endpoints and the
weight-difference sum have $L^p$ norm at most $Cv(w)$. For the last sum, the Lipschitz bound in the parameter
and~\eqref{eq:sa-elementary-moments} yield
\begin{align*}
    \left\|
        \sum_{k=1}^{n-1}w_{k+1}
        \bigl[
            PG_{\theta_{k+1}}(X_k)
            -PG_{\theta_k}(X_k)
        \bigr]
    \right\|_{L^p}
    &\lesssim
    \sum_{k=1}^{n-1}\norm{w_{k+1}}_{\mathrm{op}}
        \norm{\theta_{k+1}-\theta_k}_{L^p}\\
    &\lesssim
    \sum_{k=1}^{n-1}\norm{w_{k+1}}_{\mathrm{op}}\gamma_k
    \lesssim
    \sum_{k=1}^{n}\norm{w_k}_{\mathrm{op}}\gamma_k.
\end{align*}
Together with~\eqref{eq:sa-adaptive-martingale-bound},
this controls the weighted sum of $b_{\theta_k}(X_k)$.

Lemma~\ref{lem:Lr-Poisson} gives $g\in L^q(\pi)$ and 
$\sup_k\norm{T_k}_{L^p}<\infty$.
Hence
\[
    \sum_{k=1}^n w_k(T_{k-1}-T_k)
    =
    w_1T_0-w_nT_n
    +\sum_{k=1}^{n-1}(w_{k+1}-w_k)T_k
\]
has $L^p$ norm at most $Cv(w)$.
Finally,
\[
    \norm{\sum_{k=1}^n w_kr(\theta_k)}_{L^p}
    \lesssim
    \sum_{k=1}^n\norm{w_k}_{\mathrm{op}}\gamma_k
\]
by~\eqref{eq:sa-elementary-moments}.
Combining these estimates proves
\eqref{eq:sa-weighted-remainder-bound}.
\end{proof}

\subsection{Gaussian approximation and completion of the proof}
The identity
\[
    F(\theta_k,X_k)
    =
    h(X_k)+He_k+r(\theta_k)+b_{\theta_k}(X_k)
\]
and the Poisson decomposition
\[
    h(X_k)=D_k+T_{k-1}-T_k
\]
give
\[
    e_{k+1}
    =
    (I_d-\gamma_kH)e_k
    -\gamma_kD_k
    -\gamma_k\zeta_k.
\]
Unrolling this recursion yields
\[
    e_{n+1}
    =
    \Phi_{n,1}e_1
    -\sum_{k=1}^n\gamma_k\Phi_{n,k+1}D_k
    -\sum_{k=1}^n\gamma_k\Phi_{n,k+1}\zeta_k.
\]
Dividing by $\sqrt{\gamma_n}$, or averaging over the
terminal index and interchanging the finite sums, gives
\begin{align*}
    \frac{e_{n+1}}{\sqrt{\gamma_n}}
        =
        -\sum_{k=1}^nL_{n,k}D_k
        +\mathcal R_n^{\mathrm{last}},\qquad
    \sqrt n(\overline\theta_n-\theta^*)
        =
        -\sum_{k=1}^nR_{n,k}D_k
        +\mathcal R_n^{\mathrm{av}},
\end{align*}
where
\begin{align*}
    \mathcal R_n^{\mathrm{last}}
        =
        \gamma_n^{-1/2}\Phi_{n,1}e_1
        -\sum_{k=1}^nL_{n,k}\zeta_k,\qquad
    \mathcal R_n^{\mathrm{av}}
        =
        \frac1{\sqrt n}\sum_{j=1}^n\Phi_{j,1}e_1
        -\sum_{k=1}^nR_{n,k}\zeta_k.
\end{align*}
Combining Lemma~\ref{lem:sa-weighted-remainder} with
Lemmas~\ref{lem:sa-polynomial-weights}
and~\ref{lem:sa-harmonic-weights}, using $w_k=L_{n,k}$,
gives
\begin{equation}
\label{eq:sa-remainders-final}
    \norm{\mathcal R_n^{\mathrm{last}}}_{L^p}
    \lesssim\sqrt{\gamma_n},
\end{equation}
both for $0<a<1$ and for $a=1$ under
\eqref{eq:sa-harmonic-gap}.

For $1/2<a<1$, the same argument with $w_k=R_{n,k}$ gives
\begin{equation}
\label{eq:sa-averaged-remainder-final}
    \norm{\mathcal R_n^{\mathrm{av}}}_{L^p}
    \lesssim
    n^{-1/2}
    +\sqrt{\gamma_n}
    +\sqrt n\,\gamma_n
    \lesssim\sqrt n\,\gamma_n.
\end{equation}

\medskip
\noindent\textit{Gaussian approximation for the leading martingale.}
Let $w_k$ denote either $L_{n,k}$ or $R_{n,k}$, and put
\[
    M_n:=-\sum_{k=1}^n w_kD_k,
    \qquad
    V:=\operatorname{Cov}(M_n).
\]
The sequence $\{D_k\}_{k \geq 1}$ is stationary and square integrable,
and the martingale property gives
\[
    V=\sum_{k=1}^n w_k\Gamma w_k^\top.
\]
The deterministic covariance estimates show that $V$
converges to the relevant positive definite limit.
Consequently,
$\norm{V^{1/2}}_{\mathrm{op}}$ and
$\norm{V^{-1/2}}_{\mathrm{op}}$
are uniformly bounded for all sufficiently large $n$.

Write $b_q:=\norm{D_1}_{L^q}<\infty$, and define
\[
    \widetilde Y_{n,k}:=-V^{-1/2}w_kD_k,
    \qquad
    a_k:=\norm{\widetilde Y_{n,k}}_{L^q}.
\]
Then
\[
    a_k
    \leq
    b_q\norm{V^{-1/2}}_{\mathrm{op}}
        \norm{w_k}_{\mathrm{op}}.
\]
These are precisely the matrix-weighted Poisson martingale
increments in the main setup, with deterministic weights
\[
    M_{n,k}:=-V^{-1/2}w_k.
\]
The response-weight estimates imply
\[
    \norm{(a_k)_{k=1}^n}_2\lesssim1,
    \qquad
    \norm{(a_k)_{k=1}^n}_4^2
    \lesssim
    \begin{cases}
        \sqrt{\gamma_n},&w_k=L_{n,k},\\
        n^{-1/2},&w_k=R_{n,k}.
    \end{cases}
\]
Theorem~\ref{thm:main}, followed by the linear map
$V^{1/2}$, therefore gives
\begin{equation}
\label{eq:sa-weighted-ga-final}
    \W_p\bigl(
        \law(M_n),\mathcal N(0,V)
    \bigr)
    \lesssim
    \begin{cases}
        \sqrt{\gamma_n},&w_k=L_{n,k},\\
        n^{-1/2},&w_k=R_{n,k}.
    \end{cases}
\end{equation}

\medskip
\noindent\textit{Comparison of Gaussian covariances.}
If $V,V_*$ have eigenvalues bounded below by a positive
constant, then
\begin{equation}
\label{eq:sa-gaussian-covariance-comparison}
    \W_p\bigl(
        \mathcal N(0,V),\mathcal N(0,V_*)
    \bigr)
    \lesssim\norm{V-V_*}_{\HS}.
\end{equation}
For $0<a<1$, combine
\eqref{eq:sa-remainders-final},
\eqref{eq:sa-weighted-ga-final},
and~\eqref{eq:sa-proof-covariance-estimates} to obtain
\begin{align*}
    \W_p\!\left(
        \law\!\left(
            \frac{e_{n+1}}{\sqrt{\gamma_n}}
        \right),
        \mathcal N(0,\Sigma)
    \right)\lesssim
    \sqrt{\gamma_n}
    +\sqrt{\gamma_n}
    +\gamma_n
    +\frac1{n\gamma_n}\lesssim
    \sqrt{\gamma_n}+\frac1{n\gamma_n}.
\end{align*}
This proves part~\textup{(a)}.

For $a=1$ under~\eqref{eq:sa-harmonic-gap},
Lemma~\ref{lem:sa-harmonic-weights} instead gives
\[
    \norm{V_n-\Sigma_1}_{\HS}\lesssim n^{-1}.
\]
Together with $\sqrt{\gamma_n}\asymp n^{-1/2}$,
the same argument yields an $O(n^{-1/2})$ bound and
proves part~\textup{(b)}.

For $1/2<a<1$, equations
\eqref{eq:sa-averaged-remainder-final},
\eqref{eq:sa-weighted-ga-final},
and~\eqref{eq:sa-proof-covariance-estimates} give
\[
    \W_p\!\left(
        \law\!\left(
            \sqrt n(\overline\theta_n-\theta^*)
        \right),
        \mathcal N(0,\overline\Sigma)
    \right)
    \lesssim
\frac1{n\gamma_n}
    +\sqrt n\,\gamma_n.
\]
This proves part~\textup{(c)} and completes the proof of Theorem \ref{thm:sa-application}.

\section{Lower Bounds for the Stochastic Approximation Application}
\label{sec:sa-lower-bounds}

We establish sharpness of the Polyak--Ruppert averaging and
harmonic last-iterate rates in Theorem~\ref{thm:sa-application}.
Throughout this section, $\lesssim$ hides constants that may
depend on the fixed model parameters, the stepsize exponent,
and the moment order, but not on the iteration indices.
All examples use i.i.d. observations $X_k=\xi_k$.

\subsection{Polyak--Ruppert averaging}

\begin{proposition}[Sharpness for Polyak--Ruppert averaging]
\label{prop:sa-pr-lower-bound}
Fix $p\geq2$ and $a\in(1/2,1)$.
There exists a scalar SA model satisfying
Assumptions~\ref{assumption:MC} and
\ref{ass:sa-application}, with $\gamma_k=k^{-a}$, such
that, for all sufficiently large $n$,
\begin{equation}
\label{eq:sa-pr-polynomial-lower}
    \W_p\!\left(
        \law\!\left(
            \sqrt n(\overline\theta_n-\theta^*)
        \right),
        \mathcal N(0,\overline\Sigma)
    \right)
    \geq c_a\bigl(n^{1/2-a}+n^{a-1}\bigr),
\end{equation}
where $c_a>0$ is independent of $n$.
The model may depend on $a$.
\end{proposition}

\begin{proof}
We give separate examples for the nonlinear-bias term
$n^{1/2-a}$ and the covariance term $n^{a-1}$.

\medskip
\noindent\textit{Nonlinear bias.}
Fix $0<\kappa,\sigma<1/4$.
Let $f(u)=u+\kappa u^2$ on $[-1,1]$, and extend $f$
linearly outside this interval, matching its value and
derivative at each endpoint. Then $f$ is globally
Lipschitz, has the unique zero $0$, and satisfies
\[
    1-2\kappa\leq f'(u)\leq1+2\kappa.
\]
Let $\{\xi_k\}_{k \geq 1}$ be i.i.d. noise variables taking the values
$\pm\sigma$ with equal probability, and consider
\begin{equation}
\label{eq:sa-lower-nonlinear-model}
    \theta_{k+1}
    =\theta_k-\gamma_k f(\theta_k)+\gamma_k\xi_k,
    \qquad
    \theta_1=0,
    \qquad
    \gamma_k=k^{-a}.
\end{equation}
This model has $F(\theta,x)=f(\theta)-x$,
$\theta^*=0$, $H=1$, and $\Gamma=\overline\Sigma=\sigma^2$.
The derivative bounds imply
\[
    |u-\gamma f(u)|
    \leq
    \bigl[1-(1-2\kappa)\gamma\bigr]|u|,
    \qquad 0<\gamma\leq1.
\]
Since $\sigma<1-2\kappa$, induction gives
$|\theta_k|\leq1$ for every $k$.
Consequently, the recursion is exactly
\[
    \theta_{k+1}
    =(1-\gamma_k)\theta_k
    -\kappa\gamma_k\theta_k^2+\gamma_k\xi_k.
\]
To verify moment stability, a conditional even-moment
expansion and induction on the integer $m\geq1$ give
\[
    \E|\theta_{k+1}|^{2m}
    \leq
    \bigl[
        1-2m(1-2\kappa)\gamma_k+C_m\gamma_k^2
    \bigr]\E|\theta_k|^{2m}
    +C_m\gamma_k^{m+1},
\]
where $C_m$ is independent of $k$.
The final term uses the moment bounds at orders below
$2m$, starting from the zeroth moment.
Since
\[
    \frac{\gamma_k^m/\gamma_{k+1}^m-1}{\gamma_k}
    \longrightarrow0,
\]
comparison with $K\gamma_k^m$ for sufficiently large $K$
yields $\E|\theta_k|^{2m}\lesssim\gamma_k^m$.
Thus
\begin{equation}
\label{eq:sa-lower-nonlinear-moments}
    \norm{\theta_k}_{L^r}\lesssim\sqrt{\gamma_k}
    \qquad\text{for every fixed }r\geq1.
\end{equation}
The same comparison applies to $\gamma_k=1/k$, since
$2m(1-2\kappa)>m$.
Hence the example satisfies the moment-stability condition
for both stepsize choices.

For the bias lower bound, taking expectations gives
$\E\theta_k\leq0$ and
\begin{align}
\label{eq:sa-lower-bias-recursions}
    -\E\theta_{k+1}
        &=(1-\gamma_k)(-\E\theta_k)
          +\kappa\gamma_k\E\theta_k^2,\\
    \E\theta_{k+1}^2
        &\geq
          (1-3\gamma_k)\E\theta_k^2
          +\sigma^2\gamma_k^2.
    \nonumber
\end{align}
For the second inequality, use $|\theta_k|\leq1$ and
\[
    \bigl[(1-\gamma_k)\theta_k
          -\kappa\gamma_k\theta_k^2\bigr]^2
    \geq
    (1-3\gamma_k)\theta_k^2.
\]
For sufficiently large $k$, iterate the second recursion
over the last $\lfloor\gamma_k^{-1}\rfloor$ steps.
On this interval the stepsizes are comparable to
$\gamma_k$, and the products of the coefficients
$1-3\gamma_j$ are bounded below by a positive constant.
Hence $\gamma_k\lesssim\E\theta_k^2$.
Applying the same argument to the first recursion gives
$\gamma_k\lesssim-\E\theta_k$.
Therefore,
\begin{equation}
\label{eq:sa-lower-pr-mean}
    \sqrt n\,\gamma_n
    \lesssim
    \frac1{\sqrt n}\sum_{k=1}^n\gamma_k
    \lesssim
    -\frac1{\sqrt n}\sum_{k=1}^n\E\theta_{k+1}
    =-\E\bigl[\sqrt n\,\overline\theta_n\bigr].
\end{equation}
Since Wasserstein distance dominates the difference
of means, this proves the $n^{1/2-a}$ lower bound.

\medskip
\noindent\textit{Covariance mismatch.}
Now let $\{\xi_k\}_{k \geq 1}$ be i.i.d. Gaussian noise with distribution
$\mathcal N(0,\sigma^2)$, and consider
\begin{equation}
\label{eq:sa-lower-linear-model}
    \theta_{k+1}
    =(1-\gamma_k)\theta_k+\gamma_k\xi_k,
    \qquad
    \theta_1=0,
    \qquad
    \gamma_k=k^{-a}.
\end{equation}
Here $H=1$ and $\Gamma=\overline\Sigma=\sigma^2$.
The iterates are centered Gaussian, and their variance
recursion gives $\E\theta_k^2\lesssim\gamma_k$.
Thus moment stability holds for every fixed moment order.

Write $s_k:=\E\theta_k^2$.
The exact recursion
\[
    s_{k+1}
    =(1-\gamma_k)^2s_k+\sigma^2\gamma_k^2
\]
implies
\begin{equation}
\label{eq:sa-lower-second-moment-expansion}
    \lim_{k\to\infty}
    k\left(s_{k+1}-\frac{\sigma^2}{2}\gamma_k\right)
    =\frac{a\sigma^2}{4}.
\end{equation}
To verify this, substitute
$\sigma^2\gamma_k/2+a\sigma^2/(4k)$ for $s_{k+1}$.
Using
\[
    \left|
        \frac{\gamma_{k-1}}{\gamma_k}-1-\frac ak
    \right|
    \lesssim k^{-2},
\]
the resulting residual is bounded by
$C(\gamma_k^3+k^{-2})$.
Iterating the contracting recursion gives
\[
    \left|
        s_{k+1}-\frac{\sigma^2}{2}\gamma_k
        -\frac{a\sigma^2}{4k}
    \right|
    \lesssim
    \gamma_k^2+\frac1{k^2\gamma_k}.
\]
Multiplication by $k$ proves
\eqref{eq:sa-lower-second-moment-expansion}, since
$k\gamma_k^2\to0$ and $k\gamma_k\to\infty$. Define
\[
    c_n:=
    \operatorname{Cov}\!\left(
        \sum_{j=1}^n\theta_{j+1},\theta_{n+1}
    \right),
    \qquad c_0:=0.
\]
Independence of the noise gives the exact identities
\[
    c_n=(1-\gamma_n)c_{n-1}+s_{n+1},
    \qquad
    \operatorname{Var}(\sqrt n\,\overline\theta_n)
    =\frac1n\sum_{k=1}^n(2c_k-s_{k+1}).
\]
Comparing the first recursion with
$\sigma^2/2+a\sigma^2/(4n\gamma_n)$, using the preceding
bound on $s_{n+1}$, gives
\[
    \left|
        c_n-\frac{\sigma^2}{2}
        -\frac{a\sigma^2}{4n\gamma_n}
    \right|
    \lesssim
    \gamma_n+\frac1{(n\gamma_n)^2}.
\]
Indeed, substitution leaves a residual bounded by
$C(\gamma_n^2+n^{a-2})$, and iteration yields the
stated inequality. Consequently,
\[
    \lim_{n\to\infty}
    n\gamma_n\left(c_n-\frac{\sigma^2}{2}\right)
    =\frac{a\sigma^2}{4}.
\]
Together with
\[
    \gamma_n\sum_{k=1}^n s_{k+1}
    \lesssim n\gamma_n^2\longrightarrow0,
    \qquad
    \gamma_n\sum_{k=1}^n\frac1{k\gamma_k}
    \longrightarrow\frac1a,
\]
this implies
\begin{equation}
\label{eq:sa-lower-pr-variance}
    \lim_{n\to\infty}
    n\gamma_n
    \left[
        \operatorname{Var}(\sqrt n\,\overline\theta_n)
        -\sigma^2
    \right]
    =\frac{\sigma^2}{2}.
\end{equation}
Because the average is centered Gaussian and $p\geq2$,
for all sufficiently large $n$,
\[
    \frac1{n\gamma_n}
    \lesssim
    \left|
        \sqrt{\operatorname{Var}(\sqrt n\,\overline\theta_n)}
        -\sigma
    \right|
    \leq
    \W_p\!\left(
        \law(\sqrt n\,\overline\theta_n),
        \mathcal N(0,\sigma^2)
    \right).
\]
This proves the $n^{a-1}$ lower bound.

Finally, use the nonlinear example for $1/2<a\leq3/4$
and the Gaussian example for $3/4<a<1$.
The selected bound is the larger of
$n^{1/2-a}$ and $n^{a-1}$, and their sum is at most
twice their maximum. This proves the assertion.
\end{proof}

\subsection{Sharpness of the harmonic rate}

\begin{proposition}[A harmonic lower bound]
\label{prop:sa-harmonic-lower-bound}
There is a scalar model satisfying all assumptions of
Theorem~\ref{thm:sa-application}(b), including
$\gamma_1H>1$, such that, for every fixed $p\geq2$,
\[
    \W_p\!\left(
        \law\!\left(
            \frac{\theta_{n+1}}{\sqrt{\gamma_n}}
        \right),
        \mathcal N(0,\Sigma_1)
    \right)
    \geq c_p n^{-1/2}
\]
for all sufficiently large $n$, where $c_p>0$ is
independent of $n$.
\end{proposition}

\begin{proof}
Consider~\eqref{eq:sa-lower-linear-model} with
$\gamma_k=2/k$ and i.i.d. noise satisfying
\[
    \mathbb P(\xi_k=\sqrt2)=\frac13,
    \qquad
    \mathbb P(\xi_k=-1/\sqrt2)=\frac23.
\]
Then $\E\xi_k=0$, $\E\xi_k^2=1$, and
$\E\xi_k^3=1/\sqrt2$.
Here $H=1$, $\gamma_1=2$, and $\Sigma_1=2/3$.
For $n\geq2$, direct multiplication of the response
factors gives
\[
    T_n:=\frac{\theta_{n+1}}{\sqrt{\gamma_n}}
    =\sum_{k=2}^n w_{n,k}\xi_k,
    \qquad
    w_{n,k}:=
    \frac{\sqrt2\,(k-1)}{\sqrt n\,(n-1)}.
\]
These weights satisfy
\[
    \sum_{k=2}^n w_{n,k}^2
    =\frac{2n-1}{3(n-1)},
    \qquad
    \sum_{k=2}^n w_{n,k}^3
    =\frac{\sqrt n}{\sqrt2\,(n-1)}.
\]
Bounded noise and the martingale moment inequality imply
$\sup_{n\geq2}\norm{T_n}_{L^r}<\infty$ for every fixed
$r\geq1$. Thus the original iterates satisfy the
moment-stability condition. Moreover,
\[
    \E T_n^3
    =\frac1{\sqrt2}\sum_{k=2}^n w_{n,k}^3
    =\frac{\sqrt n}{2(n-1)}.
\]
Let $Z\sim\mathcal N(0,2/3)$.
For any coupling of $T_n$ and $Z$, Cauchy--Schwarz gives
\begin{align*}
    |\E T_n^3|
    &\leq
    \norm{T_n-Z}_{L^2}
    \bigl(
        \norm{T_n}_{L^4}^2
        +\norm{T_n}_{L^4}\norm{Z}_{L^4}
        +\norm{Z}_{L^4}^2
    \bigr)\\
    &\lesssim\norm{T_n-Z}_{L^2}.
\end{align*}
Taking the infimum over couplings and using
$\W_p\geq\W_2$ proves the result.
\end{proof}

\subsection{Obstructions below and at the harmonic threshold}

\begin{proposition}[Why stability alone does not give the harmonic rate]
\label{prop:sa-harmonic-gap-obstructions}
The condition $\gamma_1H>1$ in the scalar harmonic theorem
cannot be replaced by $\gamma_1H>1/2$ without changing its
general conclusion.
More precisely, for every $\beta\in(1/2,1)$ there is a
model satisfying the other assumptions, with
$\gamma_k=1/k$ and $H=\beta$, for which
\[
    \W_p\!\left(
        \law(\sqrt n\,\theta_{n+1}),
        \mathcal N(0,\Sigma_1)
    \right)
    \geq c_\beta n^{1/2-\beta}.
\]
There is also a model with $\gamma_k=1/k$ and $H=1$ for
which the same distance is at least
$c n^{-1/2}\log n$ for all sufficiently large $n$.
The constants $c_\beta,c$ are positive and independent of $n$.
\end{proposition}

\begin{proof}
For $1/2<\beta<1$, let
\[
    \theta_{k+1}
    =(1-\beta/k)\theta_k+\xi_k/k,
    \qquad \theta_1=1,
\]
where $(\xi_k)$ is i.i.d. standard Gaussian noise.
Then $\Gamma=1$ and $\Sigma_1=(2\beta-1)^{-1}$.
The variance recursion gives
$\operatorname{Var}(\theta_k)\lesssim k^{-1}$, while
\[
    n^{-\beta}
    \lesssim
    \E\theta_{n+1}
    =\prod_{k=1}^n(1-\beta/k)
    \lesssim n^{-\beta}.
\]
The product bounds follow by taking logarithms and using
convergence of
$\sum_{k\geq1}[\log(1-\beta/k)+\beta/k]$.
Gaussianity therefore gives moment stability for every
fixed moment order, and comparison of means proves the
first lower bound.

For the critical case, use the nonlinear model
\eqref{eq:sa-lower-nonlinear-model} with $\gamma_k=1/k$.
Its moment stability was established
in~\eqref{eq:sa-lower-nonlinear-moments}.
Writing $s_k:=\E\theta_k^2$, the exact recursion and
these moment bounds give
\[
    \left|
        s_{k+1}-(1-2/k)s_k-\frac{\sigma^2}{k^2}
    \right|
    \lesssim k^{-5/2}.
\]
Comparing with $\sigma^2/k$ and iterating yields
\[
    \left|s_k-\frac{\sigma^2}{k}\right|
    \lesssim k^{-3/2},
    \qquad
    \lim_{k\to\infty}k s_k=\sigma^2.
\]
The exact mean recursion gives
\[
    n\E\theta_{n+1}
    =-\kappa\sum_{k=1}^n s_k,
    \qquad
    \lim_{n\to\infty}
    \frac{-n\E\theta_{n+1}}{\log n}
    =\kappa\sigma^2.
\]
Here $H=1$, $\Gamma=\sigma^2$, and $\Sigma_1=\sigma^2$.
Consequently, for all sufficiently large $n$,
\[
    \frac{\log n}{\sqrt n}
    \lesssim
    \bigl|\E[\sqrt n\,\theta_{n+1}]\bigr|
    \leq
    \W_p\!\left(
        \law(\sqrt n\,\theta_{n+1}),
        \mathcal N(0,\sigma^2)
    \right).
\]
This proves the critical lower bound.
\end{proof}

\end{document}